\documentclass[11pt]{amsart}
\usepackage{cite}
\usepackage{graphicx}
\usepackage{latexsym,amsmath,amsfonts,amscd, amsthm}
\usepackage{color}
\usepackage[colorlinks=true]{hyperref}
\hypersetup{urlcolor=blue, citecolor=red}
\usepackage{tikz}
\usetikzlibrary{arrows,backgrounds,snakes}
\usepackage{epsfig}
\usepackage{dsfont}
\usepackage{amssymb}
\usepackage{xcolor}
\usepackage{comment}
\usepackage{subfigure}
\usepackage{multirow}
\usepackage{mathrsfs}
\usepackage{diagbox}

\newtheoremstyle{plainNoItalics}{}{}{\normalfont}{}{\bfseries}{.}{ }{}
\theoremstyle{plain}
\newtheorem{thm}{Theorem}[section]

\newtheorem{defn}[thm]{Definition}

\newtheorem{prop}[thm]{Proposition}
\newtheorem{exa}[thm]{Example}

\newcommand{\beq}{\begin{equation}}
\newcommand{\eeq}{\end{equation}}
\newcommand{\beqa}{\begin{eqnarray}}
\newcommand{\eeqa}{\end{eqnarray}}
\newcommand{\bit}{\begin{itemize}}
\newcommand{\eit}{\end{itemize}}
\newcommand{\bedef}{\begin{defn}}
\newcommand{\edefn}{\end{defn}}
\newcommand{\bpro}{\begin{prop}}
\newcommand{\epro}{\end{prop}}

\newcommand{\eps}{\varepsilon}

\newcommand{\bx}{{\bf x}}
\newcommand{\bv}{{\bf v}}

\newcommand{\bu}{{\bf u}}
\newcommand{\bn}{{\bf n}}

\newcommand{\bA}{{\bf A}}
\newcommand{\bB}{{\bf B}}

\newcommand{\bI}{{\bf I}}

\newcommand{\bV}{{\bf V}}

\title[High order AP schemes for MHD in all sonic Mach number]{High order asymptotic preserving finite difference WENO schemes with constrained transport for MHD equations in all sonic Mach numbers}


\begin{document}
	
\maketitle
\medskip
\centerline{\scshape Wei Chen}
\medskip
{\footnotesize
\centerline{School of Mathematical Sciences, Xiamen University}
\centerline{Xiamen, Fujian, 361005, P.R. China}
\centerline{Email: chenwei8@stu.xmu.edu.cn}
}

\medskip
\centerline{\scshape Kailiang Wu\footnote{This work is supported in part by NSFC grant 12171227}}
\medskip
{\footnotesize
	\centerline{Department of Mathematics \& SUSTech International Center for Mathematics}
	\centerline{Southern University of Science and Technology}
    \centerline{National Center for Applied Mathematics Shenzhen (NCAMS)}		
	\centerline{Shenzhen, Guangdong 518055, China}
	\centerline{Email: wukl@sustech.edu.cn}
}

\medskip
\centerline{\scshape Tao Xiong\footnote{Corresponding author. The work of this author was partially supported by NSFC grant No. 11971025, and the Strategic Priority Research Program of the Chinese Academy of Sciences Grant No. XDA25010401.}}
\medskip
{\footnotesize
   % please put the address of the author
\centerline{School of Mathematical Sciences, Xiamen University}
\centerline{Fujian Provincial Key Laboratory of Mathematical Modeling and High-Performance Scientific Computing}
\centerline{Xiamen, Fujian, 361005, P.R. China}
\centerline{Email: txiong@xmu.edu.cn}
}

\bigskip

\begin{abstract}
In this paper, a high-order semi-implicit (SI) asymptotic preserving (AP) and divergence-free finite difference weighted essentially nonoscillatory (WENO) scheme is proposed for magnetohydrodynamic (MHD) equations. We consider the sonic Mach number $\eps $ ranging from 
$0$ to $\mathcal{O}(1)$. High-order accuracy in time is obtained by SI implicit-explicit Runge–Kutta (IMEX-RK) time discretization. High-order accuracy in space is achieved by finite difference WENO schemes with characteristic-wise reconstructions. A constrained transport method is applied to maintain a discrete divergence-free condition. We formally prove that the scheme is AP. Asymptotic accuracy (AA) in the incompressible MHD limit is obtained if the implicit part of the SI IMEX-RK scheme is stiffly accurate. Numerical experiments are provided to validate the AP, AA, and divergence-free properties of our proposed approach. Besides, the scheme can well capture discontinuities such as shocks in an essentially non-oscillatory fashion in the compressible regime, while it is also a good incompressible solver with uniform large-time step conditions in the low sonic Mach limit.
\end{abstract}

\vspace{0.1cm}

\section{Introduction}
\label{sec1}
\setcounter{equation}{0}
\setcounter{figure}{0}
\setcounter{table}{0}
Ideal magnetohydrodynamic (MHD) equations are widely used in the modeling of weather prediction, astrophysics, as well as laboratory plasma applications such as flows in tokamaks and stellarators. Many shock capturing schemes with explicit time discretizations have been developed for solving compressible ideal MHD equations, including high order discontinuous Galerkin \cite{li2005locally,li2012arbitrary,wu2018provably,wu2022provably}, finite difference \cite{minoshima2019high,christlieb2018high,christlieb2014finite,christlieb2016high,christlieb2015positivity}, and finite volume schemes \cite{leidi2022finite,xu2016divergence,balsara2015divergence,susanto2013high,wu2021weno}, etc.

For MHD equations, divergence free of the magnetic field is a very important property. If the initial magnetic field is divergence-free, this condition will be maintained for all later times. Numerically, if this divergence-free condition is violated seriously, an unphysical force will be created to parallel the magnetic field \cite{christlieb2016high}, which can produce some numerical instabilities or nonphysical features in the computed solution, cause the loss of pressure positivity \cite{wu2018positivity,wu2019provably}, and/or lead to the failure of the simulation. The constrained transport (CT) methodology, which was originally introduced by Evans and Hawley\cite{evans1988simulation}, is an effective way to maintain the (discrete) divergence-free condition up to machine precision. Many improvements and extensions \cite{balsara1999staggered,balsara2004second,dai1998simple,dai1998divergence,christlieb2018high,christlieb2016high,christlieb2015positivity,christlieb2014finite,helzel2011unstaggered} are followed.  

Recently, the MHD system in the low sonic Mach limit has attracted a lot of interests \cite{mamashita2021slau2,leidi2022finite,dumbser2019divergence,jiang2010incompressible,jiang2012low,cui2015incompressible,minoshima2021low}. Jiang, Ju, and Li \cite{jiang2010incompressible} have shown that a weak solution of compressible MHD equations will converge to a strong solution of their corresponding incompressible MHD equations in this low sonic Mach limit. They have also investigated the low Mach limit of full MHD equations with a heat conductivity \cite{jiang2012low}. Cui, Ou, and Ren \cite{cui2015incompressible} have established a uniform convergence from full compressible MHD equations to isentropic incompressible MHD equations with well-prepared initial conditions in the three-dimensional case. 
Numerically it is also very attractive to develop schemes for MHD flows at any speed. For the above mentioned explicit shock capturing schemes, when applied to the MHD system in the low sonic Mach regime, the time step is subject to a very strict CFL condition which is proportional to the sonic Mach number, making it very undesirable for all speed flows. An implicit time discretization can release the small time step restrictions from an explicit one, but usually it results in a highly nonlinear system, which is not easy to be solved efficiently with an iterative method and it is hard to guarantee a fast convergence.  Furthermore, numerical viscosities for such schemes are inversely proportional to the sonic Mach number, which introduce excessive numerical dissipations. Instead, in recent years some researchers are in pursuit of semi-implicit (SI) schemes with low dissipations. In \cite{dumbser2019divergence},  Dumbser et. al. developed a pressure-based SI finite volume all Mach number flow solver for compressible MHD equations. Minoshima and Miyoshi \cite{minoshima2021low} proposed a multistate low-dissipation advection upstream splitting method for MHD with both high and low Mach numbers, by using a Harten-Lax-van Leer discontinuities (HLLD) approximate Riemann solver. Leidi et al. \cite{leidi2022finite} presented an SI finite volume solver with a 5-wave HLLD and a well-balanced method to efficiently simulate MHD flows at low Mach numbers with a gravitational source. 

For the MHD system with all sonic Mach numbers, it is important to design schemes with asymptotic stability and consistency as the sonic Mach number goes to $0$ in the incompressible limit, namely, asymptotic preserving (AP). AP schemes have been widely used in a variety of areas, such as hydrodynamic or diffusive limits of kinetic models, relaxation limits of hyperbolic models, and low Mach number limits of compressible fluid models. We refer readers to the papers \cite{jin1999efficient,jin2022asymptotic,jin2010asymptotic,jin2022spatial,filbet2010class,dimarco2012high,dimarco2017study} for more information. To achieve high order AP schemes in time, SI implicit-explicit Runge-Kutta (IMEX-RK) methods are widely used, which apply an explicit time discretization for  nonstiff terms, while an implicit time discretization for stiff terms \cite{boscarino2007error,boscarino2013implicit,boscarino2016high,boscarino2018implicit,boscarino2019high,boscarino2022high,kanevsky2007application,kadioglu2011second,o2017simulation,cao2016implicit,meng2020fourth,Boscheri2021high}. With a suitable choice of explicit and implicit discretizations, both high efficiency and uniform stability independent of the stiff parameter can be obtained. As related, in \cite{tavelli2017pressure} Tavelli and Dumbser developed an SI space-time discontinuous Galerkin method for all Mach-number flow, which involves staggered meshes and a Picard iteration for solving a large nonlinear system. In \cite{boscarino2019high,boscarino2022high}, Boscarino, Qiu, Russo, and Xiong presented high-order SI IMEX weighted essentially nonoscillatory (WENO) schemes for all-Mach isentropic Euler equations and all-Mach full Euler system, respectively. In \cite{Boscheri2021high}, Boscheri and Pareschi developed high order pressure-based semi-implicit IMEX schemes for the three dimensional Navier-Stokes equations at all Mach numbers. In \cite{huang2022high}, Huang, Xing, and Xiong designed a high-order well-balanced SI AP scheme for shallow water equations in all Fraude numbers.
We refer to many other related works from the references therein.

In this paper, we would like to develop a high-order AP finite difference WENO scheme with CT for MHD equations in all sonic Mach numbers. To our best knowledge, there are very few AP schemes for MHD systems in literature. The new contributions and innovations of this work are outlined as follows:
\begin{itemize}
	\item In the low sonic Mach number regime, SI IMEX-RK methods are adopted for time discretization to obtain a uniform time step independent of the sonic Mach number. Besides, to avoid a nonlinearity from the equation of state (EOS), an SI approach %similar to the one in \cite{boscarino2022high} 
	is used, leading 
	to a linearized elliptic equation for the pressure. With carefully designed explicit and implicit discretizations, only a linear system needs to be solved. The resulting scheme is more efficient than explicit shock capturing schemes, especially in the low sonic Mach regime.
	\item 
	Numerical viscosities are carefully designed to avoid excessive numerical dissipations, and a CT method is applied to maintain a discrete divergence-free property. The scheme can well capture discontinuities such as shocks in an essentially non-oscillatory fashion in the compressible regime. Meanwhile, the scheme is also a good incompressible solver as $\eps\rightarrow 0$. Due to less dissipation, our scheme can be shown to perform better than explicit schemes for low sonic Mach problems.
	\item
	Formal AP and asymptotically accurate (AA) properties in the stiff limit as the sonic Mach number $\eps \rightarrow 0$ are proved, by assuming the implicit part of an SI IMEX-RK scheme is stiffly accurate (SA).
\end{itemize}

The rest of the paper is as follows. In Section \ref{sec2}, we will briefly review the ideal MHD system, corresponding to its low sonic Mach limit based on asymptotic expansions, and a CT methodology for the divergence-free condition. A first-order SI scheme in time will first be introduced in Section \ref{sec3}, and then a high-order SI scheme with IMEX-RK in time and finite difference WENO in space is followed. Formal proofs of AP and  AA properties are given in Section \ref{sec4}. Numerical experiments are performed in Section \ref{sec5}. A brief conclusion will be drawn in the last section.

%%%%%%%%%%%%%%%%%%%%%%%%%%%%%
%%%%%%%%%%%%%%%%%%%%%%%%%%%%%
%%%%%%%%%%%%%%%%%%%%%%%%%%%%%
%%%%%%%%%%%%%%%%%%%%%%%%%%%%%
%%%%%%%%%%%%%%%%%%%%%%%%%%%%%
%%%%%%%%%%%%%%%%%%%%%%%%%%%%%

\section{Equations of compressible ideal MHD}
\label{sec2}
\setcounter{equation}{0}
\setcounter{figure}{0}
\setcounter{table}{0}

The ideal MHD equations in a conservative form can be written as
\begin{equation}
\frac{\partial }{\partial t}
\begin{bmatrix}
    \rho  \\ 
    \rho \mathbf{u}\\  
    \mathbf{B}   \\ 
    E     
\end{bmatrix}
+
\nabla \cdot
\begin{bmatrix}
    \rho \mathbf{u} \\  
    \rho \mathbf{u} \otimes \mathbf{u} + (p + \frac{1}{2} {\Vert \mathbf{B} \Vert}^2) \mathbf{I} - \mathbf{B} \otimes \mathbf{B} \\  
    \mathbf{u} \otimes \mathbf{B} - \mathbf{B} \otimes \mathbf{u}\\  
    (E + p + \frac{1}{2} {\Vert \mathbf{B} \Vert}^2) \mathbf{u} - \mathbf{B} (\mathbf{u} \cdot \mathbf{B})
\end{bmatrix}
=
\mathbf{0},
  \label{S2_E1}
\end{equation}
where $\rho$ is the fluid mass density, $\mathbf{u} = (u, v, w)$ is the velocity, $\mathbf{B} = (B_x, B_y, B_z)$ is the magnetic field, $E$ is the total energy density, $p$ is the gas pressure, and $\Vert \cdot \Vert$ is the Euclidean vector norm. The total energy is given by
\begin{equation}
    E = \frac{p}{\gamma -1} + \frac{1}{2} (\rho {\Vert \mathbf{u} \Vert}^2 + {\Vert \mathbf{B} \Vert}^2),
    \label{S2_E3}
\end{equation}
where $\gamma$ is the specific heat ratio. The system is subject to a divergence-free condition of the magnetic field
\begin{equation}
    \nabla \cdot \mathbf{B} = 0.
    \label{S2_E2}
\end{equation} 
If initially, the divergence-free condition holds, from \eqref{S2_E1}, it also holds for all later times \cite{christlieb2014finite}. For this reason, \eqref{S2_E2} is usually not regarded as a constraint (like the $\nabla \cdot \bu = 0$ constraint for the incompressible Navier-Stokes equations), but rather an involution \cite{helzel2011unstaggered,christlieb2014finite,christlieb2016high}.
%%%%%%%%%%%%%%%%%%%%%%%%%%%%%
%%%%%%%%%%%%%%%%%%%%%%%%%%%%%
%%%%%%%%%%%%%%%%%%%%%%%%%%%%%
\subsection{Eigenstructure of the MHD system}
For \eqref{S2_E1}, the eigenvalues of the Jacobian for the flux function along a normal direction $\bn$ are given as:
\begin{equation}
    \lambda_{1,8} = \bu \cdot \bn \mp c_f, \quad
    \lambda_{2,7} = \bu \cdot \bn \mp c_a, \quad
    \lambda_{3,6} = \bu \cdot \bn \mp c_s, \quad
    \lambda_{4,5} = \bu \cdot \bn, \quad
    \label{S2_E4}
\end{equation}
where $\lambda_{1,8}$, $\lambda_{2,7}$, $\lambda_{3,6}$, $\lambda_{4}$, and $\lambda_{5}$ are the left/right fast
magnetosonic waves, left/right Alfvén waves, left/right slow magnetosonic waves, entropy wave, and divergence wave, correspondingly. The sound speed $a$, Alfvén speed $c_a$, fast magnetosonic speed $c_f$, and slow magnetosonic speed $c_s$ are defined as 
\begin{equation}
    \left\{
    \begin{aligned}
    a &:= \sqrt{\frac{\gamma p}{\rho}}, \\
    c_a &:= \sqrt{\frac{{(\bB \cdot \bn)}^2}{\rho}}, \\
    c_{f,s} &:= {\left\{ \frac{1}{2} \left[ a^2 + \frac{{\Vert \bB \Vert}^2}{\rho} \pm \sqrt{{\left(a^2 + \frac{{\Vert \bB \Vert}^2}{\rho}\right)}^2 - 4a^2 {c^2_a} } \right] \right\}}^{\frac{1}{2}}.
    \end{aligned}\right.
    \label{S2_E5}
\end{equation}
A fixed ordering for the MHD eigenvalues is
\begin{equation*}
    \lambda_1 \leq \lambda_2 \leq \lambda_3 \leq \lambda_4 \leq \lambda_5 \leq \lambda_6 \leq \lambda_7 \leq \lambda_8.
\end{equation*}
Spectral decomposition for this system has already been described in \cite{jiang1999high,brio1988upwind,cargo1997roe,powell1997approximate,powell1999solution,takahashi2013regular}, and we will follow the one in \cite{jiang1999high} for all simulations in the current paper.
%%%%%%%%%%%%%%%%%%%%%%%%%%%%%
%%%%%%%%%%%%%%%%%%%%%%%%%%%%%
%%%%%%%%%%%%%%%%%%%%%%%%%%%%%
\subsection{Non-dimensionalized MHD}
To derive a low sonic Mach limit for the ideal MHD system, we start with rewriting \eqref{S2_E1}, \eqref{S2_E3}, and \eqref{S2_E2} into a dimensionless form. We choose some reference or characteristic values, such as a length $x_0$, a time $t_0$, a fluid mass density $\rho_0$, a velocity $u_0$, a magnetic field $B_0$, and a gas pressure $p_0$, with $u_0 = x_0 / t_0$. For the sake of simplicity, we consider a regime where the characteristic Alfvén Mach number $u_0 \sqrt{\rho_0} / B_0 = 1$. Then we define the following dimensionless variables:
\begin{equation*}
  \hat{{\bx}} = \frac{{\bx}}{x_0}, \quad
  \hat{t} = \frac{t}{t_0}, \quad
  \hat{\rho} = \frac{\rho}{\rho_0}, \quad
  \hat{\bu} = \frac{\bu}{u_0}, \quad
  \hat{\bB} = \frac{\bB}{B_0}, \quad
  \hat{p} = \frac{p}{p_0}.
\end{equation*}
Inserting these new variables into \eqref{S2_E1}, we obtain the following non-dimensionalised MHD equations\cite{leidi2022finite}:
\begin{equation}
\frac{\partial }{\partial t}
\begin{bmatrix}
    \rho  \\
    \rho \mathbf{u}\\
    \mathbf{B}   \\
    E     \\
\end{bmatrix}
+
\nabla \cdot
\begin{bmatrix}
    \rho \mathbf{u}\\
    \rho \mathbf{u} \otimes \mathbf{u} + (\frac{p}{\eps^2} + \frac{1}{2} {\Vert \mathbf{B} \Vert}^2) \mathbf{I} - \mathbf{B} \otimes \mathbf{B}\\
    \mathbf{u} \otimes \mathbf{B} - \mathbf{B} \otimes \mathbf{u}\\
    (E + p + \frac{\eps^2}{2} {\Vert \mathbf{B} \Vert}^2) \mathbf{u} - \eps^2 \mathbf{B} (\mathbf{u} \cdot \mathbf{B})
\end{bmatrix}
=
\mathbf{0},
\label{S2_E8}
\end{equation}
where we have dropped the hats of the dimensionless variables for ease of presentation. The EOS, divergence-free condition and the characteristic sonic Mach number $\eps$ are now given by
\begin{equation}
    E := \frac{p}{\gamma - 1} + \frac{\eps^2}{2} (\rho {\Vert \bu \Vert}^2 + {\Vert \bB \Vert}^2), \quad
    \nabla \cdot \bB = 0, \quad\eps := u_0 \sqrt{\frac{\rho_0}{p_0}}.
    \label{S2_E9}
\end{equation}
In this dimensionless MHD system \eqref{S2_E8} with \eqref{S2_E9}, the expressions of those eigenvalues \eqref{S2_E4} remain unchanged, except the sound speed $a$, Alfvén speed $c_a$, fast magnetosonic speed $c_f$, and slow magnetosonic speed $c_s$ from \eqref{S2_E5} now become
\begin{equation*}
    \left\{
    \begin{aligned}
    a &:= \frac{1}{\eps} \sqrt{\frac{\gamma p}{\rho}}, \\
    c_a &:= \sqrt{\frac{{(\bB \cdot \bn)}^2}{\rho}}, \\
    c_{f,s} &:= {\left\{ \frac{1}{2} \left[ a^2 + \frac{{\Vert \bB \Vert}^2}{\rho} \pm \sqrt{{ \left(a^2 + \frac{{\Vert \bB \Vert}^2}{\rho} \right)}^2 - 4a^2 {c^2_a} } \right] \right\}}^{\frac{1}{2}}.
    \end{aligned}\right.
\end{equation*}
Following the eigenstructure of the dimensionless MHD system discussed above, it is natural to find that if the background flow velocity $u_0$ is
slow, then the sonic Mach number $\eps = u_0 \sqrt{\rho_0} / \sqrt{p_0} \ll 1$, and the fast magnetosonic speed $c_f$ becomes very large. We can see that large magnetosonic speed $c_f$ will make the fast magnetosonic waves $\lambda_{1,8}$ to be much faster as compared to other waves. In this situation, for explicit shock-capturing schemes, the time step subjects to
\begin{equation*}
    \Delta t = {\rm CFL} \frac{\Delta x}{\mathop{\max}\limits_{1 \leq i \leq 8} \vert \lambda_i \vert} =  \mathcal{O} (\eps \Delta x),
\end{equation*}
where $\Delta t$ is the time step size, $\Delta x$ is the mesh size, and ${\rm CFL}$ is the time stability CFL number. On one hand, highly inaccurate solutions will appear due to the excessive numerical viscosity (scales as $\eps^{-1}$) in standard upwind schemes. On the other hand, this restriction results in an increasingly large computational cost for low sonic Mach fluid flows. A straight approach is to apply an implicit time discretization to avoid such severe time step constraints. However, a fully implicit scheme will result in a complicated nonlinear system which is usually very hard to be solved efficiently. Besides, it may not be able to guarantee a correct asymptotic limit. Thus, it is important to design efficient schemes with both asymptotic stability and consistency in the low sonic Mach limit, namely, the AP property.
%%%%%%%%%%%%%%%%%%%%%%%%%%%%%
%%%%%%%%%%%%%%%%%%%%%%%%%%%%%
%%%%%%%%%%%%%%%%%%%%%%%%%%%%%
\subsection{Low sonic Mach limit}
\label{2.3}
Next, we recall a formal derivation of the incompressible MHD system from the non-dimensionalized compressible MHD equations \eqref{S2_E8} with \eqref{S2_E9}. We consider an asymptotic expansion ansatz for the following variables \cite{matthaeus1988nearly}:
\begin{equation}
    \left\{
    \begin{aligned}
    &p(\mathbf{x},t) = p_0(\mathbf{x},t) + \eps p_1(\mathbf{x},t) + \eps^2 p_2(\mathbf{x},t) + \cdots,\\
    &\rho(\mathbf{x},t) = \rho_0(\mathbf{x},t) + \eps \rho_1(\mathbf{x},t) + \eps^2 \rho_2(\mathbf{x},t) + \cdots,\\
    &\mathbf{u}(\mathbf{x},t) = \mathbf{u}_0(\mathbf{x},t) + \eps \mathbf{u}_1(\mathbf{x},t) + \eps^2 \mathbf{u}_2(\mathbf{x},t) + \cdots,\\
    &\mathbf{B}(\mathbf{x},t) = \mathbf{B}_0(\mathbf{x},t) + \eps \mathbf{B}_1(\mathbf{x},t) + \eps^2 \mathbf{B}_2(\mathbf{x},t) + \cdots.
    \end{aligned}\right.
    \label{S2_E12}
\end{equation}
Inserting \eqref{S2_E12} into \eqref{S2_E8} and \eqref{S2_E9}, equating to zero for different orders of $\eps$, corresponding to the first three leading orders, we have 
\bit
\item $\mathcal{O}(\eps^{-2})$
\begin{equation*}
    \nabla p_0 = 0,
\end{equation*}
\item $\mathcal{O}(\eps^{-1})$
\begin{equation*}
    \nabla p_1 = 0,
\end{equation*}
\item $\mathcal{O}(\eps^{0})$
\begin{equation}
\frac{\partial }{\partial t}
\begin{bmatrix}
    \rho_0  \\
    \rho_0 \mathbf{u}_0\\
    \mathbf{B}_0   \\
    E_0     \\
\end{bmatrix}
+
\nabla \cdot
\begin{bmatrix}
    \rho_0 \mathbf{u}_0\\
    \rho_0 \mathbf{u}_0 \otimes \mathbf{u}_0 + (p_2 + \frac{1}{2} {\Vert \mathbf{B}_0 \Vert}^2) \mathbf{I} - \mathbf{B}_0 \otimes \mathbf{B}_0\\
    \mathbf{u}_0 \otimes \mathbf{B}_0 - \mathbf{B}_0 \otimes \mathbf{u}_0\\
    (E_0 + p_0) \mathbf{u}_0
\end{bmatrix}
=
\mathbf{0},
\label{S2_E15}
\end{equation}
\eit
where we have assumed a similar expansion of $E$ as $p$ in \eqref{S2_E12}, and take $E_0 = p_0 / (\gamma - 1)$. Denoting the material derivative ${\rm d/d}t = \partial / \partial t + \bu_0 \cdot \nabla$, from the energy equation in \eqref{S2_E15}, $E_0 = p_0 / (\gamma - 1)$, with $\nabla p_0 = 0$, we have
\begin{equation}
    \nabla \cdot \bu_0 = - \frac{1}{p_0 \gamma} \frac{{\rm d} p_0}{{\rm d}t}.
    \label{S2_E16}
\end{equation}
Taking into account periodic or no-slip $\bu \cdot \bn = 0$ boundary conditions on the spatial domain $\Omega$, and integrating \eqref{S2_E16} over $\Omega$, we get $\int_{\Omega} \nabla \cdot \bu_0 d \bx = \int_{\partial \Omega} \bu_0 \cdot \bn d s = 0$, where $\bn$ is the unit outward normal vector
along $\partial \Omega$. This implies $p_0$ is constant not only in space but also in time, that is $p_0 = {\rm Const}$. A direct result $\nabla \cdot \bu_0 = 0$ can be obtained through \eqref{S2_E16} with $p_0 = {\rm Const}$. Inserting $\nabla \cdot \bu_0 = 0$ into the mass-continuity equation in \eqref{S2_E15}, we get ${\rm d} \rho_0 / {\rm d}t = 0$ which means that the density is constant along the characteristic trajectories. 

For the order $\mathcal{O} (\eps)$ in \eqref{S2_E8}, we have: 
$$\partial E_1 / \partial t + \nabla \cdot [(E_0 + p_0) \bu_1 + (E_1 + p_1) \bu_0] = 0.$$ 
Due to $E_1 = p_1 / (\gamma - 1)$ from the EOS in \eqref{S2_E9}, similarly as above, we get $p_1 = {\rm Const}$ and $\nabla \cdot \bu_1 = 0$. In this case, we include $\eps \mathbf{u}_1$ into $\mathbf{u}_0$, and $\eps p_1$ into $p_0$, the asymptotic expansions of $\bu$ and $p$ become 
\begin{equation}
	\mathbf{u} = \mathbf{u}_0 + \eps^2 \mathbf{u}_2 + \cdots,
	\label{ep_u}
\end{equation}
and 
\begin{equation}
	p = p_0 + \eps^2 p_2 + \cdots.
	\label{ep_p}
\end{equation}
Using such expansions for $\mathbf{u}$ and $p$, as $\eps \rightarrow 0$, the leading order of the system \eqref{S2_E8} yields
\begin{equation}
    \left\{
    \begin{aligned}
    &\frac{\partial }{\partial t}
    \begin{bmatrix}
        \rho_0 \mathbf{u}_0\\
        \mathbf{B}_0   \\
    \end{bmatrix}
    +
    \nabla \cdot
    \begin{bmatrix}
        \rho_0 \mathbf{u}_0 \otimes \mathbf{u}_0 + (p_2 + \frac{1}{2} {\Vert \mathbf{B}_0 \Vert}^2) \mathbf{I} - \mathbf{B}_0 \otimes \mathbf{B}_0\\
        \mathbf{u}_0 \otimes \mathbf{B}_0 - \mathbf{B}_0 \otimes \mathbf{u}_0
    \end{bmatrix}
    =
    \mathbf{0},\\
    &\frac{{\rm d} \rho_0}{{\rm d}t} = 0, \quad
    p_0 = {\rm Const}, \quad
    \nabla \cdot \bu_0 = 0, \quad
     \nabla \cdot \bB = 0.\\
    \end{aligned}\right.
    \label{S2_E17}
\end{equation}

Numerically, in order to have a correct asymptotic limit from \eqref{S2_E8} to \eqref{S2_E17}, it is important to impose well-prepared initial conditions, which are consistent with the expansions given in \eqref{S2_E12}, with $\bu$ and $p$ from \eqref{ep_u} and \eqref{ep_p}. According to \cite{dellacherie2010analysis,klein1995semi}, the well-prepared initial conditions for \eqref{S2_E8} are given as
\begin{equation}
    \left\{
    \begin{aligned}
    &\rho(t = 0, \bx) = \rho_0 (\bx) + \mathcal{O} (\eps) > 0,\quad
    &&p(t = 0, \bx) = p_0 + \eps^2 p_2(\bx) + \cdots, \\
    &\bu(t = 0, \bx) = \bu_0 (\bx) + \mathcal{O} (\eps^2),\quad
    &&\bB(t = 0, \bx) = \bB_0(\bx) + \mathcal{O} (\eps),    
    \end{aligned}\right.
    \label{S2_E18}
\end{equation}
where $\rho_0 (\bx)$ is a strictly positive function, $p_0 = {\rm Const}$, $\nabla \cdot \bu_0 = 0$, and $\nabla \cdot \bB_0 = 0$.
%%%%%%%%%%%%%%%%%%%%%%%%%%%%%
%%%%%%%%%%%%%%%%%%%%%%%%%%%%%
%%%%%%%%%%%%%%%%%%%%%%%%%%%%%
\subsection{CT for $\nabla \cdot \bB = 0$}
Many CT methods have been proposed to satisfy the discrete divergence-free condition $\nabla \cdot \bB = 0$ for the ideal MHD equations \cite{christlieb2014finite,christlieb2015positivity,christlieb2016high,christlieb2018high}. Here we consider the one described in \cite{christlieb2014finite}, in which the magnetic field is written as the curl of a magnetic vector potential $\bA$:
\begin{equation}
    \bB = \nabla \times \bA.
    \label{S2_E19}
\end{equation}
Due to $\nabla \cdot (\bu \otimes \bB - \bB \otimes \bu) = \nabla \times (\bB \times \bu)$, the equation of the magnetic field $\bB$ in the MHD system can be rewritten as
\begin{equation}
    \frac{\partial \bB}{\partial t} + \nabla \times (\bB \times \bu) = 0.
    \label{S2_E20}
\end{equation}
Inserting \eqref{S2_E19} into \eqref{S2_E20}, we further get
\begin{equation*}
    \nabla \times \left(\frac{\partial \bA}{\partial t} + (\nabla \times \bA) \times \bu \right) = 0,
\end{equation*}
which implies the existence of a scalar function $\psi$ such that
\begin{equation*}
    \frac{\partial \bA}{\partial t} + (\nabla \times \bA) \times \bu = -\nabla \psi.
\end{equation*}
Stable solutions can be obtained by introducing a Weyl gauge, i.e. $\psi \equiv 0$ \cite{helzel2011unstaggered}, and the evolution equation for the vector potential becomes 
\begin{equation*}
    \frac{\partial \bA}{\partial t} + (\nabla \times \bA) \times \bu = 0.
\end{equation*}
It is easy to check that $\nabla \cdot \bB = \nabla \cdot (\nabla \times \bA) = 0$, i.e. the divergence-free condition will be satisfied for all times. In summary, we get the following MHD system:
\begin{equation}
    \left\{
    \begin{aligned}
    &\frac{\partial \rho}{\partial t} + \nabla \cdot \left(\rho \bu \right) = 0,\\
    &\frac{\partial \rho \bu }{\partial t} + \nabla \cdot \left(\rho \mathbf{u} \otimes \mathbf{u} + \left(\frac{p}{\eps^2} + \frac{1}{2} {\Vert \mathbf{B} \Vert}^2\right) \mathbf{I} - \mathbf{B} \otimes \mathbf{B}\right) = \mathbf{0},\\
    &\frac{\partial \bA}{\partial t} + (\nabla \times \bA) \times \bu = \mathbf{0},\\
    &\frac{\partial E}{\partial t} + \nabla \cdot \left( \left(E + p + \frac{\eps^2}{2} {\Vert \mathbf{B} \Vert}^2 \right) \mathbf{u} - \eps^2 \mathbf{B} \left(\mathbf{u} \cdot \mathbf{B} \right) \right) = 0,\\
    &\bB - \nabla \times \bA = \mathbf{0}.\\
    \end{aligned}\right.
    \label{S2_E24}
\end{equation}
Corresponding to \eqref{S2_E24}, with well-prepared initial conditions \eqref{S2_E18}, the low sonic Mach limit incompressible MHD system becomes
\begin{equation}
    \left\{
    \begin{aligned}
    \label{S2_E26}
    &\frac{d\rho_0}{dt}=0, \quad p_0 = {\rm Const}, \quad
    \nabla \cdot \bu_0 = 0, \quad \nabla \cdot \mathbf{B}_0 = 0,\\
    &\frac{\partial \rho_0 \mathbf{u}_0}{\partial t} + \nabla \cdot \left(  \rho_0 \mathbf{u}_0 \otimes \mathbf{u}_0 + \left(p_2 + \frac{1}{2} {\Vert \mathbf{B}_0 \Vert}^2 \right) \mathbf{I} - \mathbf{B}_0 \otimes \mathbf{B}_0 \right) = \mathbf{0},\\
    &\frac{\partial \bA_0}{\partial t} + (\nabla \times \bA_0) \times \bu_0 = \mathbf{0},\\
    &\bB_0 - \nabla \times \bA_0 = \mathbf{0}.
    \end{aligned}\right.
\end{equation}

%%%%%%%%%%%%%%%%%%%%%%%%%%%%%
%%%%%%%%%%%%%%%%%%%%%%%%%%%%%
%%%%%%%%%%%%%%%%%%%%%%%%%%%%%
%%%%%%%%%%%%%%%%%%%%%%%%%%%%%
%%%%%%%%%%%%%%%%%%%%%%%%%%%%%
%%%%%%%%%%%%%%%%%%%%%%%%%%%%%

\section{Numerical schemes}
\label{sec3}
\setcounter{equation}{0}
\setcounter{figure}{0}
\setcounter{table}{0}

In this section, we will construct and analyze a high-order finite difference scheme with AP and divergence-free properties for the non-dimensionalized MHD system \eqref{S2_E24}. The features of our scheme include the following: we design an SI IMEX-RK time discretization so that the scheme is stable with a time step constraint independent of the sonic Mach number $\eps$; the AP property is preserved in the zero sonic Mach number limit; the scheme can be implemented efficiently without solving a nonlinear system. We will adopt the high order finite difference WENO scheme in \cite{boscarino2022high,boscarino2019high,huang2022high}, and the CT method in \cite{christlieb2014finite}. The final scheme can well capture discontinuities such as shocks  in the compressible regime, while keeping a discrete divergence-free condition in the low sonic Mach regime, with a uniform time step condition independent of $\eps$. We start with a first-order SI scheme. The high-order IMEX-RK scheme in time and a finite difference WENO reconstruction in space will be followed. After that, an overall scheme will be summarized.
%%%%%%%%%%%%%%%%%%%%%%%%%%%%%
%%%%%%%%%%%%%%%%%%%%%%%%%%%%%
%%%%%%%%%%%%%%%%%%%%%%%%%%%%%
\subsection{First order SI IMEX scheme}
\label{3.1}
In the following, we 
start with a first-order SI IMEX time discretization for \eqref{S2_E24}, while keeping space continuous:
\begin{subequations}
    \label{S3_A1}
\begin{equation}
    \frac{\rho^{n+1} - \rho^n}{\Delta t} + \nabla \cdot \mathbf{q}^n = 0,
    \label{S3_E1}
\end{equation}
\begin{equation}
    \frac{\mathbf{q}^{n+1} - \mathbf{q}^n}{\Delta t} + \nabla \cdot \left(\frac{\mathbf{q}^n \otimes \mathbf{q}^n}{\rho^n} - \mathbf{B}^n \otimes \mathbf{B}^n +\left(\frac{1}{2} {\Vert \mathbf{B}^n \Vert}^2 + p^n \right) \mathbf{I}\right) + \frac{1-\eps^2}{\eps^2} \nabla p^{n+1} = \mathbf{0},
    \label{S3_E2}
\end{equation}
\begin{equation}
    \frac{\mathbf{A}^{n+1} - \mathbf{A}^n}{\Delta t} + \left(\nabla \times \bA^{n}\right) \times \frac{\mathbf{q}^n}{\rho^n} = \mathbf{0},
    \label{S3_E3}
\end{equation}
\begin{equation}
    \frac{E^{n+1} - E^n}{\Delta t} + \nabla \cdot \left(\frac{E^n + p^n}{\rho^{n+1}} \mathbf{q}^{n+1} +\eps^2 \left(\frac{1}{2}{\Vert \mathbf{B}^n \Vert}^2 \frac{\mathbf{q}^n}{\rho^n} - \left(\frac{\mathbf{q}^n}{\rho^n} \cdot \mathbf{B}^n \right)\mathbf{B}^n \right) \right) = 0,
    \label{S3_E4}
\end{equation}
\begin{equation}
    \bB^{n+1} - \nabla \times \bA^{n+1} = \mathbf{0},
    \label{S3_E5}
\end{equation}
\end{subequations}
where $\mathbf{q} = \rho \bu$. The updating flow chart based on the SI IMEX scheme \eqref{S3_A1} is as follows:
\begin{itemize}
\item \textbf{Step 1.} Update $\rho^{n+1}$ and $\bA^{n+1}$ from \eqref{S3_E1} and \eqref{S3_E3} respectively, then update $\bB^{n+1}$ from \eqref{S3_E5}.
    
\item \textbf{Step 2.} We rewrite \eqref{S3_E2} as
\begin{equation}
    \mathbf{q}^{n+1} = \mathbf{q}^* - \Delta t \frac{1-\eps^2}{\eps^2} \nabla p^{n+1},
    \label{S3_E6}
\end{equation}
with
\begin{equation*}
    \mathbf{q}^* = \mathbf{q}^n -\Delta t \nabla \cdot \left(\frac{\mathbf{q}^n \otimes \mathbf{q}^n}{\rho^n} - \mathbf{B}^n \otimes \mathbf{B}^n +\left(\frac{1}{2} {\Vert \mathbf{B}^n \Vert}^2 + p^n\right) \mathbf{I}\right). 
\end{equation*}
Substituting $\mathbf{q}^{n+1}$ into \eqref{S3_E4}, we can further get
\begin{equation}
    E^{n+1} = E^* + \Delta t^2 \frac{1-\eps^2}{\eps^2} \nabla \cdot (H^n \nabla p^{n+1}),
    \label{S3_E7}
\end{equation}
where $H^n = (E^n + p^n) / \rho^{n+1}$ and 
\[ 
E^* = E^n -\Delta t \nabla \cdot (H^n \mathbf{q}^*) -\Delta t \eps^2 \nabla \cdot \left(\frac{1}{2}{\Vert \mathbf{B}^n \Vert}^2 \frac{\mathbf{q}^n}{\rho^n} - \left(\frac{\mathbf{q}^n}{\rho^n} \cdot \mathbf{B}^n\right)\mathbf{B}^n\right).
\]
Corresponding to the pressure in the incompressible limit, we introduce a pressure perturbation
$p^{n+1}_2$, defined as 
\begin{equation}
p^{n+1}_2 = (p^{n+1} - \bar{p}^n)/ \eps^2,
\label{p2}
\end{equation} 
where $\bar{p}^n = \int_{\Omega} p^n d \bx / \vert \Omega \vert$, and we replace $E^{n+1}$ by $p^{n+1}/ (\gamma -1) + \eps^2 ({\Vert \mathbf{q}^n \Vert}^2 / (2\rho^n) + {\Vert \mathbf{B}^n \Vert}^2 /2 )$ from \eqref{S3_E7}. We rewrite \eqref{S3_E7} as
\begin{equation}
    \frac{\eps^2 }{\gamma -1} p^{n+1}_2 - \Delta t^2 (1-\eps^2) \nabla \cdot (H^n \nabla p^{n+1}_2) = E^{**},
    \label{S3_E8}
\end{equation}
where $E^{**} = E^* - \bar{p}^n/ (\gamma -1 ) - \eps^2 ({\Vert 
 \mathbf{q}^n \Vert}^2 / (2\rho^n) + {\Vert \mathbf{B}^n \Vert}^2 /2 )$. We can now update $p_2^{n+1}$ from the elliptic equation \eqref{S3_E8}.

\item \textbf{Step 3.} We update $\mathbf{q}^{n+1}$ from \eqref{S3_E6} by using \eqref{p2}, and then $E^{n+1}$ from \eqref{S3_E4}.
\end{itemize}
As described in \cite{boscarino2022high}, we may set $\nabla \cdot (\alpha p^n \bI)$ and $(1-\alpha \eps^2) \nabla p^{n+1} /\eps^2 $ to replace $\nabla \cdot (p^n \bI)$ and $(1- \eps^2) \nabla p^{n+1} /\eps^2 $ in \eqref{S3_E2} with $\alpha = 1$ for $\eps < 1$ and $\alpha = 1 / \eps^2$ for $\eps \geq 1$. Note that if $\eps \geq 1$ the implicit pressure contribution in \eqref{S3_E2}  vanishes, so the momentum $\mathbf{q}^{n+1}$ is evaluated explicitly. With updated $\rho^{n+1}$, $\bA^{n+1}$, $\bB^{n+1}$, and $\mathbf{q}^{n+1}$, $E^{n+1}$ can also be updated in an explicit way from \eqref{S3_E4}. In one dimension, the divergence-free condition is satisfied automatically, so there is no need to apply a CT method. We can replace \eqref{S3_E3} and \eqref{S3_E5} with 
\begin{equation}
    \frac{\mathbf{B}^{n+1} - \mathbf{B}^n}{\Delta t} + \nabla \cdot \left(\frac{\mathbf{q}^n \otimes \mathbf{B}^n}{\rho^n} - \frac{\mathbf{B}^n \otimes \mathbf{q}^n}{\rho^n}\right) = 0.
    \label{S3_replace}
\end{equation}
The first step in the flow chart becomes: update $\rho^{n+1}$ and $\bB^{n+1}$ from \eqref{S3_E1} and \eqref{S3_replace} respectively, and other steps remain unchanged.
%%%%%%%%%%%%%%%%%%%%%%%%%%%%%
%%%%%%%%%%%%%%%%%%%%%%%%%%%%%
%%%%%%%%%%%%%%%%%%%%%%%%%%%%%
\subsection{High order SI IMEX-RK scheme}
\label{3.2}
In this subsection, we will follow a similar procedure as in \cite{boscarino2022high,boscarino2019high} to extend the first-order SI IMEX scheme \eqref{S3_A1} to high order. For simplicity, we only consider \eqref{S2_E8} with CT methods, namely, \eqref{S2_E24}. 

First, \eqref{S2_E24} can be written as an autonomous system:
\begin{equation*}
    \left\{
    \begin{aligned}
    &U_t = \mathcal{H} (U,U,\bB), \\
    &\bB = \nabla \times \bA,
    \end{aligned}\right.
\end{equation*}
where $U = {(\rho, \mathbf{q}, \bA, E)}^T$ and $\mathcal{H}: \mathbb{R}^n \times \mathbb{R}^n \times \mathbb{R}^m \rightarrow \mathbb{R}^n$ is a sufficiently regular mapping. We use subscript ``E" to express an explicit treatment for the first argument $U$ and the third one $\bB$, and subscript ``I" to express an implicit treatment for the second $U$.
% We solve
% \begin{equation}
%     \left\{
%     \begin{aligned}
%     &U^{'}_E = \mathcal{H} (U_E, U_I, \bB_E),\\
%     &U^{'}_I = \mathcal{H} (U_E, U_I, \bB_E),\\
%     &\bB_E = \nabla \times \bA_E,
%     \label{a2}
%     \end{aligned}\right.
% \end{equation}
% with initial conditions
% \begin{equation}
%     U_E(t_0) = U_0, \quad
%     U_I(t_0) = U_0, \quad
%     \bB_E(t_0) = \bB_0.
%     \label{a3}
% \end{equation}
For the system \eqref{S2_E24}, the function $\mathcal{H} (U_E,U_I,\bB_E)$ is defined as
\begin{equation*}
    \begin{aligned}
    \mathcal{H} (U_E,U_I,\bB_E) = 
    &- \nabla \cdot
    \begin{pmatrix}
	\mathbf{q}_E \\
    \frac{\mathbf{q}_E \otimes \mathbf{q}_E}{\rho_E} - \mathbf{B}_E \otimes \mathbf{B}_E +(\frac{1}{2} {\Vert \mathbf{B}_E \Vert}^2 + p_E) \mathbf{I} \\
    \mathbf{0} \\
    0 \\
	\end{pmatrix}
    -
    \begin{pmatrix}
	0 \\
    \mathbf{0} \\
    (\nabla \times \bA_E) \times \frac{\mathbf{q}_E}{\rho_E}\\
    0
	\end{pmatrix} \\
    &- \nabla \cdot
    \begin{pmatrix}
	0 \\
    \mathbf{0} \\
    \mathbf{0} \\
    \eps^2 (\frac{1}{2}{\Vert \mathbf{B}_E \Vert}^2 \frac{\mathbf{q}_E}{\rho_E} - (\frac{\mathbf{q}_E}{\rho_E} \cdot \mathbf{B}_E)\mathbf{B}_E)
	\end{pmatrix}
    - \nabla \cdot
    \begin{pmatrix}
	0 \\
    (1-\eps^2)p_{I,2} \mathbf{I} \\
    \mathbf{0} \\
    H_I \mathbf{q}_I
	\end{pmatrix} \\ \, \\
    =& -\nabla \cdot \mathcal{F}_E^1 - \mathcal{G}_{E} - \nabla \cdot 
\mathcal{F}_E^2 - \nabla \cdot \mathcal{F}_{SI},
    \end{aligned}
\end{equation*}
with $H_I = (E_E + p_E)/\rho_I$ and $p_{I,2} = (p_I - \bar{p}_E) / \eps^2$.

To obtain high-order accuracy in time, we can apply a multi-stage IMEX-RK time discretization with a double Butcher $tableau$,
\begin{equation*}\label{DBT}
    \begin{array}{c|c}
    \tilde{c} & \tilde{A}\\
    \hline
    \vspace{-0.25cm}
    \\
    & \tilde{b}^T \end{array} \ \ \ \ \ \qquad
    \begin{array}{c|c}
    {c} & {A}\\
    \hline
    \vspace{-0.25cm}
    \\
    & {b^T} \end{array},
\end{equation*}
where $\tilde{A} = (\tilde{a}_{ij})$ is an $s \times s$ matrix for an explicit scheme, with $\tilde{a}_{ij}=0$ for $j \geq i$, and $A = ({a}_{ij})$ is an $s \times s$ matrix for an implicit scheme. For the implicit part of the methods, we use a diagonally implicit scheme, i.e. $a_{ij}=0$, for $j > i$, to guarantee simplicity and efficiency in solving the algebraic equations corresponding to the implicit part of the discretization. The vectors $\tilde{c}=(\tilde{c}_1,...,\tilde{c}_s)^T$, $\tilde{b}=(\tilde{b}_1,...,\tilde{b}_s)^T$, $c=(c_1,...,c_s)^T$, and $b=(b_1,...,b_s)^T$ complete the characterization of the scheme. The coefficients $\tilde{c}$ and $c$ are given by the usual relation
\begin{eqnarray*}\label{eq:candc}
    \tilde{c}_i = \sum_{j=1}^{i-1} \tilde a_{ij}, \ \ \ c_i = \sum_{j=1}^{i} a_{ij}.
\end{eqnarray*}
Later, for the consideration of AP property, we consider implicit schemes with SA property, i.e., the implicit part of the Butcher table satisfies the condition $b^T = e_s^TA$, with $e_s = (0, \cdots 0,1)^T$. We will see that the SA property guarantees that the numerical solution is identical to the last internal stage value of the scheme. Now, we may update the solutions as follows. Starting from $U_E^{(0)}=U_I^{(0)}=U^n$ and $\bB_E^{(0)} = \bB^n$, for inner stages $i = 1  \text{ to }  s$:
\bit
\item First update the solution $U_E^{(i)}$ for the explicit part
\begin{equation}
  \label{a7}
  U_E^{(i)} = U^n + \Delta t\sum^{i-1}_{j=1}\tilde{a}_{ij}\mathcal{H}(U_E^{(j)},U_I^{(j)},\bB_E^{(j)}), \quad \bB_E^{(i)} = \nabla \times \bA_E^{(i)}.
\end{equation}
\item Update the known values for the implicit part $U_*^{(i)}$, where
  \begin{equation}
  \label{a8}
  U_*^{(i)} = U^n + \Delta t\sum^{i-1}_{j=1}a_{ij}\mathcal{H}(U_E^{(j)},U_I^{(j)},\bB_E^{(j)}),
  \end{equation}
and then solve
   \begin{equation}
   \label{a9}
   U_I^{(i)} = U_*^{(i)} + \Delta ta_{ii}\mathcal{H}(U_E^{(i)},U_I^{(i)},\bB_E^{(i)}).
   \end{equation}
\item Finally, compute the solutions $U^{n+1}$ and $\bB^{n+1}$ at time level $t^{n+1}$ from
\begin{equation}
\label{a10}
U^{n+1}  = U^n + \Delta t \sum_{i=1}^s b_i \mathcal{H}(U_E^{(i)},U_I^{(i)},\bB_E^{(i)}), \quad
\bB^{n+1} = \nabla \times \bA^{n+1}.
\end{equation}
\eit
The first order scheme \eqref{S3_A1} is the same as applying the following Butcher table
to \eqref{S2_E24}:
\begin{equation*}\label{Afirst}
\begin{array}{c|c}
0 & 0 \\
\hline
&1
\end{array} \qquad\qquad
\begin{array}{c|c}
1 & 1     \\
\hline
&  1
\end{array},
\end{equation*}
namely $U_E = U^n$ and $U_I = U^{n+1}$.

\subsection{High order finite difference WENO scheme} Below we will describe our spatial discretization strategies which incorporate a WENO mechanism in order to capture shocks in the compressible regime and produce a high order incompressible solver for the flow in the zero Mach limit. In particular, we propose to apply a fifth order characteristic-wise WENO procedure to $\mathcal{F}_{E}^1$ of the ideal MHD system so that in the compressible regime numerical oscillations could be well controlled, and to apply a component-wise WENO procedure for the flux functions of $\mathcal{F}_{E}^2$ and $\mathcal{F}_{SI}$. As for $\mathcal{G}_{E}$, which is a Hamiltonian, a WENO procedure described in \cite{jiang2000weighted} will be applied.
\bit
\item Component-wise WENO $\nabla_{W}$ for $\nabla \cdot \mathcal{F}_{E}^2$ and $\nabla \cdot \mathcal{F}_{SI}$:\\
Let $\bv = (\rho, \mathbf{q}, \bB, E)^T$ and $\mathbf{f} = \mathbf{f}(\bv)$ be the unknown variables and flux of the ideal MHD system \eqref{S2_E1}. For simplicity, we consider the 1D case as an example, and we take a uniform mesh with $N+1$ grid points: $a = x_0 < x_1 < \cdots < x_N = b$. A conservative finite difference WENO scheme for system \eqref{S2_E1} can be written in the following flux-difference form:
\begin{equation*}
    (\mathbf{v}_i)_t = \frac{1}{\Delta x}(\hat{\mathbf{f}}_{i+\frac{1}{2}} - \hat{\mathbf{f}}_{i-\frac{1}{2}}),
\end{equation*}
where $\hat{\mathbf{f}}_{i \pm \frac{1}{2}}$ are the numerical fluxes, which can be reconstructed by the following procedures:
\begin{enumerate}
	\item Compute the physical flux at each grid point: $\mathbf{f}_i = \mathbf{f}(\bv_i)$.\\
\item Perform a Lax–Friedrichs flux vector splitting for each component of the physical variables to get
\begin{equation}
	\label{LF}
    \mathbf{f}^{\pm}_i = \frac{1}{2}(\mathbf{f}_i \pm \alpha \mathbf{v}_i),
\end{equation}
with $\alpha = \mathop{\max}\limits_{0 \leq k \leq N } (\vert \bu_k \vert + \hat{c}_f^{(k)})$, where
\begin{equation}
    \label{3.21}
    \left\{
    \begin{aligned}
    &\hat{a}^{(k)} := {\rm min}\left(\frac{1}{\eps},1 \right) \sqrt{\frac{\gamma p_k}{\rho_k}}, \quad
    \hat{c}_a^{(k)} := \sqrt{\frac{{(\bB_k \cdot \bn)}^2}{\rho_k}}, \\
    &\hat{c}_{f,s}^{(k)} := {\left\{ \frac{1}{2} \left[ \left(\hat{a}^{(k)} \right)^2 + \frac{{\Vert \bB_k \Vert}^2}{\rho_k} \pm \sqrt{{\left(\left(\hat{a}^{(k)} \right)^2 + \frac{{\Vert \bB_k \Vert}^2}{\rho_k}\right)}^2 - \left(2{\hat{a}^{(k)}} {\hat{c}^{(k)}_a} \right)^2 } \right] \right\}}^{\frac{1}{2}}.
    \end{aligned}\right.
\end{equation}
\item Perform a finite difference WENO reconstruction to obtain upwind and downwind numerical fluxes:
\begin{equation*}
    \hat{\mathbf{f}}^{+}_{i+\frac{1}{2}} = \Phi_{\rm WENO5} (\mathbf{f}^{+}_{i-2},\mathbf{f}^{+}_{i-1},\mathbf{f}^{+}_{i},\mathbf{f}^{+}_{i+1},\mathbf{f}^{+}_{i+2}), \quad \hat{\mathbf{f}}^{-}_{i+\frac{1}{2}} = \Phi_{\rm WENO5}(\mathbf{f}^{-}_{i+3},\mathbf{f}^{-}_{i+2},\mathbf{\mathbf{f}}^{-}_{i+1},\mathbf{f}^{-}_{i},\mathbf{f}^{-}_{i-1}).
\end{equation*}
Here $\Phi_{\rm WENO5}$ denotes a 5th-order WENO reconstruction \cite{jiang2000weighted}, while other orders can also be used. Then we set
\begin{equation*}
    \hat{\mathbf{f}}_{i+\frac{1}{2}} = \hat{\mathbf{f}}^{+}_{i+\frac{1}{2}} +\hat{\mathbf{f}}^{-}_{i+\frac{1}{2}}.
\end{equation*}
\end{enumerate}
% by letting $R_{i+\frac{1}{2}} = L_{i+\frac{1}{2}} = \bI$, where $\bI$ is the identity matrix. For more details, we refer the readers to the
% paper \cite{shu1998essentially}.\\

\item Characteristic-wise WENO $\nabla_{CW}$ for $\nabla \cdot \mathcal{F}_{E}^1$:\\
The Jacobian matrix for the flux function  $\partial \mathbf{f} / \partial \mathbf{v}$, has a spectral decomposition of the form $\partial \mathbf{f} / \partial \mathbf{v} = RDL$ \cite{christlieb2014finite}, where $D$ is a diagonal matrix of real eigenvalues, $R$ is the matrix of right eigenvectors, and $L = R^{-1}$ is the matrix of left eigenvectors. Then, for a characteristic-wise WENO reconstruction $\nabla_{CW}$, we have the following procedures:
\begin{enumerate}
\item Compute the physical flux at each grid point: $\mathbf{f}_i = \mathbf{f}(\bv_i)$.
\item Compute the average state $\bv_{i+\frac{1}{2}}$ by using an arithmetic mean:
\begin{equation*}
    \bv_{i+\frac{1}{2}} = \frac{1}{2}(\bv_i + \bv_{i+1}).
\end{equation*}
\item Compute the right and left eigenvectors of the flux Jacobian matrix $\partial \mathbf{f} / \partial \mathbf{v}$ by taking $\bv=\bv_{i+\frac{1}{2}}$ at $x = x_{i+\frac{1}{2}}$:
\begin{equation*}
    R_{i+\frac{1}{2}} = R(\bv_{i+\frac{1}{2}}), \quad L_{i+\frac{1}{2}} = L(\bv_{i+\frac{1}{2}}).
    \label{b1}
\end{equation*}
\item Project the solution and physical flux into the characteristic space:
\begin{equation*}
    \mathbf{w}_j = L_{i+\frac{1}{2}} \bv_j, \quad \mathbf{g}_j = L_{i+\frac{1}{2}} \mathbf{f}_j,
\end{equation*}
for all $j$ in the numerical stencil associated with $x = x_{i+\frac{1}{2}}$. In the case of a 5th-order finite difference WENO scheme: $j = i-2,i-1,i,i+1,i+2,i+3$.\\
\item Perform a Lax–Friedrichs flux vector splitting for each component of the characteristic variables to compute
\begin{equation*}
    \mathbf{g}^{\pm}_j = \frac{1}{2}(\mathbf{g}_j \pm \alpha \mathbf{w}_j),
\end{equation*}
where $\alpha$ has the same expression in the component-wise WENO \eqref{LF}.
\item Perform a WENO reconstruction $\Phi_{\rm WENO5}$ on each component of $\mathbf{g}^{\pm}_j$ to obtain their corresponding numerical fluxes, and then set
\begin{equation*}
    \hat{\mathbf{g}}_{i+\frac{1}{2}} = \hat{\mathbf{g}}^{+}_{i+\frac{1}{2}} +\hat{\mathbf{g}}^{-}_{i+\frac{1}{2}}.
\end{equation*}
\item Project the numerical flux $\hat{\mathbf{g}}_{i+\frac{1}{2}}$ back to the physical space:
\begin{equation*}
    \hat{\mathbf{f}}_{i+\frac{1}{2}} = R_{i+\frac{1}{2}} \hat{\mathbf{g}}_{i+\frac{1}{2}}.
\end{equation*}
\end{enumerate}
However, when the sonic Mach number $\eps$ is small, the fast eigenvalue $c_f$ will make such a spectral characteristic decomposition to be unstable. Since a characteristic decomposition is mainly required in the compressible regime where $\eps=\mathcal{O}(1)$, while it does not make any significant differences for small Mach numbers; in this case we directly use the eigenvectors in the expression of \eqref{S2_E1} by taking $\eps=1$, which works well from our numerical results.

\item A Hamiltonian WENO reconstruction WENO-HJ $(\cdot)_{HJ}$ for $\mathcal{G}_{E}$:\\
We take a 1D Hamilton–Jacobi equation as an example
\begin{equation}
    v_t + \mathcal{HT} (t,x,v,v_x) = 0,
    \label{b2}
\end{equation}
where $\mathcal{HT}$ is the Hamiltonian. As described in \cite{jiang2000weighted}, a semi-discrete form for \eqref{b2} is
\begin{equation*}
    v_{t,i} = -\widehat{\mathcal{HT}} (t,x_i,v_i,v^-_{x,i},v^+_{x,i}),
\end{equation*}
where $\widehat{\mathcal{HT}}$ is a numerical Hamiltonian which approximates \eqref{b2}, and $v^-_{x,i},v^+_{x,i}$ are upwind and downwind approximations of $v_x$ at $x = x_i$. A Lax-Friedrichs Hamiltonian from \cite{jiang2000weighted} is given as follows
\begin{equation}
    \widehat{\mathcal{HT}}\left(t,x,v,u^-,u^+\right) = \mathcal{HT}\left(t,x,v,\frac{u^- + u^+}{2} \right) - \beta \left(\frac{u^+ -u^-}{2} \right), 
    \label{nH}
\end{equation}
where
\begin{equation*}
    \beta = \mathop{\max}\limits_{u \in I (u^-,u^+)} \vert \mathcal{HT}_u (t,x,v,u) \vert,
\end{equation*}
and $I(u^-,u^+)$ is the interval between $u^-$ and $u^+$.
The values of $v^+_{x,i}$ and $v^-_{x,i}$ are obtained from
\begin{equation*}
    \begin{aligned}
    v^+_{x,i} &= \Phi_{\rm WENO5} \left(\frac{\Delta^+ v_{i+2}}{\Delta x},\frac{\Delta^+ v_{i+1}}{\Delta x},\frac{\Delta^+ v_{i}}{\Delta x},\frac{\Delta^+ v_{i-1}}{\Delta x},\frac{\Delta^+ v_{i-2}}{\Delta x}\right),\\
    v^-_{x,i} &= \Phi_{\rm WENO5} \left(\frac{\Delta^+ v_{i-3}}{\Delta x},\frac{\Delta^+ v_{i-2}}{\Delta x},\frac{\Delta^+ v_{i-1}}{\Delta x},\frac{\Delta^+ v_{i}}{\Delta x},\frac{\Delta^+ v_{i+1}}{\Delta x}\right),
    \end{aligned}
\end{equation*}
with $\Delta^+v_j = v_{j+1} - v_j$. Such a Hamiltonian WENO reconstruction helps us to control unphysical oscillations both in $v_x$ and $v$. For $\mathcal{G}_{E}$ in 2D, the Hamiltonian is
\begin{equation*}
    \mathcal{HT} \left(t,x,y,\frac{\partial A_z}{\partial x},\frac{\partial A_z}{\partial y}\right) = u(t,x,y) \frac{\partial A_z}{\partial x} + v(t,x,y) \frac{\partial A_z}{\partial y},
\end{equation*} 
and the numerical Hamiltonian \eqref{nH} can be applied along each direction for $u(t,x,y) \frac{\partial A_z}{\partial x}$ and $v(t,x,y) \frac{\partial A_3}{\partial y}$, respectively. Here $A_z$ is the third component of $\bA$.
\eit

\subsection{Flowchart of a high order SI IMEX-RK finite difference WENO scheme} With the above space and time discretizations, we now summarize our high order scheme as follows:
\begin{itemize}
\item \textbf{Step 1.} Starting from $U^n$ at time $t^n$, for an intermediate stage $i$ ($1 \le i \le s$), we first compute $U^{(i)}_E$ from \eqref{a7}
\begin{equation*}
\rho^{(i)}_E = \rho^n - \Delta t \sum^{i-1}_{j = 1} \Tilde{a}_{ij} \nabla_{CW} \cdot \mathbf{q}^{(j)}_E,
\end{equation*}
\begin{equation*}
\mathbf{q}^{(i)}_E = \mathbf{q}^n - \Delta t \sum^{i-1}_{j = 1} \Tilde{a}_{ij} \left(\nabla_{CW} \cdot \left(\frac{\mathbf{q}^{(j)}_E \otimes \mathbf{q}^{(j)}_E}{\rho^{(j)}_E} - \mathbf{B}^{(j)}_E \otimes \mathbf{B}^{(j)}_E +\left(\frac{1}{2} {\Vert \mathbf{B}^{(j)}_E \Vert}^2 + p^{(j)}_E \right) \mathbf{I} \right) + (1-\eps^2)\nabla_W p^{(j)}_{I,2} \right),
\end{equation*}
\begin{equation*}
\mathbf{A}^{(i)}_E = \mathbf{A}^n - \Delta t \sum^{i-1}_{j = 1} \Tilde{a}_{ij} \left( \left(\nabla \times \bA^{(j)}_E \right) \times \frac{\mathbf{q}^{(j)}_E}{\rho^{(j)}_E} \right)_{HJ},
\end{equation*}
\begin{equation*}
E^{(i)}_E = E^n - \Delta t \sum^{i-1}_{j = 1} \Tilde{a}_{ij} \left(\nabla_W \cdot \left(\eps^2 \left(\frac{1}{2}{\Vert \mathbf{B}^{(j)}_E \Vert}^2 \frac{\mathbf{q}^{(j)}_E}{\rho^{(j)}_E} - \left(\frac{\mathbf{q}^{(j)}_E}{\rho^{(j)}_E} \cdot \mathbf{B}^{(j)}_E \right)\mathbf{B}^{(j)}_E\right) \right) + \nabla_W \cdot \left(\bar{H}^{(j)} \mathbf{q}^{(j)}_I \right) \right),
\end{equation*}
with $\bar{H}^{(j)} = (E_E^{(j)}+p_E^{(j)}) / \rho_I^{(j)}$. Then we obtain 
\begin{equation*}
    \bB_E^{(i)} = \nabla \times \bA_E^{(i)},
\end{equation*}
with a 4th-order central difference discretization.

\item \textbf{Step 2.} We compute $U^{(i)}_*$ in \eqref{a8}
% \begin{subequations}
\begin{equation*}
\rho^{(i)}_* = \rho^n - \Delta t \sum^{i-1}_{j = 1} a_{ij} \nabla_{CW} \cdot \mathbf{q}^{(j)}_E,
\end{equation*}
\begin{equation*}
\mathbf{q}^{(i)}_* = \mathbf{q}^n - \Delta t \sum^{i-1}_{j = 1} a_{ij} \left(\nabla_{CW} \cdot \left(\frac{\mathbf{q}^{(j)}_E \otimes \mathbf{q}^{(j)}_E}{\rho^{(j)}_E} - \mathbf{B}^{(j)}_E \otimes \mathbf{B}^{(j)}_E +\left(\frac{1}{2} {\Vert \mathbf{B}^{(j)}_E \Vert}^2 + p^{(j)}_E \right) \mathbf{I}\right) + (1-\eps^2) \nabla_W p^{(j)}_{I,2}\right),
\end{equation*}
% \begin{equation}
% \mathbf{A}^{(i)}_* = \mathbf{A}^n - \Delta t \sum^{i-1}_{j = 1} a_{ij} \left( \left(\nabla \times \bA^{(j)}_E \right) \times \frac{\mathbf{q}^{(j)}_E}{\rho^{(j)}_E} \right)_{HJ},
% \end{equation}
\begin{equation*}
E^{(i)}_* = E^n - \Delta t \sum^{i-1}_{j = 1} a_{ij} \left(\nabla_W \cdot \left(\eps^2 \left(\frac{1}{2}{\Vert \mathbf{B}^{(j)}_E \Vert}^2 \frac{\mathbf{q}^{(j)}_E}{\rho^{(j)}_E} - \left(\frac{\mathbf{q}^{(j)}_E}{\rho^{(j)}_E} \cdot \mathbf{B}^{(j)}_E \right)\mathbf{B}^{(j)}_E\right) \right) + \nabla_W \cdot \left(\bar{H}^{(j)} \mathbf{q}^{(j)}_I \right) \right).
\end{equation*}
% \end{subequations}

\item \textbf{Step 3.} Solve $U_I^{(i)}$ from \eqref{a9}.
\begin{enumerate}
\item In components, $U_I^{(i)}$ satisfies
\begin{subequations}
\label{ccc0}
\begin{equation}
\rho^{(i)}_I = \rho^{(i)}_* - \Delta t  a_{ii} \nabla_{CW} \cdot \mathbf{q}^{(i)}_E,
\label{c1_1}
\end{equation}
\begin{equation}
\mathbf{q}^{(i)}_I = \mathbf{q}^{(i)}_{**} - \Delta t  a_{ii} (\frac{1-\eps^2}{\eps^2} \nabla p^{(i)}_I),
\label{c1_2}
\end{equation}
% \begin{equation}
% \mathbf{A}^{(i)}_I = \mathbf{A}^{(i)}_* - \Delta t a_{ii} \left( \left(\nabla \times \bA^{(i)}_E \right) \times \frac{\mathbf{q}^{(i)}_E}{\rho^{(i)}_E} \right)_{HJ},
% \label{c1_3}
% \end{equation}
\begin{equation}
E^{(i)}_I = E^{(i)}_{**} - \Delta t a_{ii} \nabla \cdot (\bar{H}^{(i)} \mathbf{q}^{(i)}_I),
\label{c1_4}
\end{equation}
\label{c1}
\end{subequations}
where
\begin{equation*}
    \mathbf{q}^{(i)}_{**} = \mathbf{q}^{(i)}_* - \Delta t  a_{ii} \nabla_{CW} \cdot \left(\frac{\mathbf{q}^{(i)}_E \otimes \mathbf{q}^{(i)}_E}{\rho^{(i)}_E} - \mathbf{B}^{(i)}_E \otimes \mathbf{B}^{(i)}_E +\left(\frac{1}{2} {\Vert \mathbf{B}^{(i)}_E \Vert}^2 + p^{(i)}_E\right) \mathbf{I} \right),
\end{equation*}
\begin{equation*}
    E^{(i)}_{**} = E^{(i)}_* - \Delta t a_{ii} \nabla_W \cdot \left(\eps^2 \left(\frac{1}{2}{\Vert \mathbf{B}^{(i)}_E \Vert}^2 \frac{\mathbf{q}^{(i)}_E}{\rho^{(i)}_E} - \left(\frac{\mathbf{q}^{(i)}_E}{\rho^{(i)}_E} \cdot \mathbf{B}^{(i)}_E \right)\mathbf{B}^{(i)}_E \right) \right).
\end{equation*}
\item We can obtain $\rho_I^{(i)}$ from \eqref{c1_1} directly. To solve \eqref{c1_2} and \eqref{c1_4}, we substitute $\mathbf{q}^{(i)}_I$ of \eqref{c1_2} into \eqref{c1_4}, and obtain
\begin{equation}
E^{(i)}_I = E^{\circ}_I + \left(1-\eps^2 \right)\Delta t^2 a^2_{ii} \nabla \cdot \left(\bar{H}^{(i)} \left(\frac{\nabla p^{(i)}_{I}}{\eps^2}\right)  \right),
\label{c3}
\end{equation}
with $E_I^{(i)}$ following the EOS 
$E_I^{(i)} = p_I^{(i)} / (\gamma -1) + \eps^2 ({\Vert\mathbf{q}_E^{(i)} \Vert}^2 / (2 \rho_E^{(i)}) + {\Vert\mathbf{B}_E^{(i)} \Vert}^2/2)$ and $E^{\circ}_I = E^{(i)}_{**} - \Delta t a_{ii} \nabla_W \cdot (\bar{H}^{(i)} \mathbf{q}^{(i)}_{**})$. By substituting 
\begin{equation}
    p^{(i)}_I = \bar{p}^{(i)}_E + \eps^2 p^{(i)}_{I,2} 
    \label{p12}
\end{equation}
into \eqref{c3}, where $\bar{p}^{(i)}_E = \int_{\Omega} p_E^{(i)} d \bx / \vert \Omega \vert$, we solve
\begin{equation}
\frac{\eps^2 }{\gamma -1 } p^{(i)}_{I,2} - \left(1-\eps^2 \right) \Delta t^2 a^2_{ii} \nabla \cdot \left(\bar{H}^{(i)} \left(\nabla p^{(i)}_{I,2}\right) \right) = E^{\circ \circ}_I,
\label{c4}
\end{equation}
with $E^{\circ \circ}_I = E^{\circ}_I - \bar{p}^{(i)}_E / (\gamma -1 ) - \eps^2  ( {\Vert \mathbf{q}^{(i)}_E \Vert}^2 / (2 \rho^{(i)}_E) + {\Vert \mathbf{B}^{(i)}_E \Vert}^2 /2 )$. Notice that in the process of substitution to obtain \eqref{c3} or \eqref{c4}, the gradient and the divergence operators are kept to be continuous, obtaining the second order operator $\nabla \cdot (\bar{H}^{(i)} \nabla p^{(i)}_{I,2})$.  The second-order spatial derivative is then discretized by a compact discretization as proposed in \cite{boscarino2019high}.
\item With $p^{(i)}_{I,2}$ solved from \eqref{c4}, we can update $\mathbf{q}^{(i)}_{I}$ from
\begin{equation*}
    \mathbf{q}^{(i)}_I = \mathbf{q}^{(i)}_{**} - \Delta t  a_{ii} \left( \left(1-\eps^2 \right) \nabla_W p^{(i)}_{I,2} \right),
\end{equation*}
and successively update $E^{(i)}_{I}$ from
\begin{equation*}
    E^{(i)}_I = E^{(i)}_{**} - \Delta t a_{ii} \nabla_W \cdot \left(\bar{H}^{(i)} \mathbf{q}^{(i)}_I \right).
\end{equation*}
\end{enumerate}

\item \textbf{Step 4.} Finally, we update the numerical solution $U^{n+1} = U^{(s)}_I$, and $\bB^{n+1} = \nabla \times \bA^{n+1}$.
\end{itemize}

%%%%%%%%%%%%%%%%%%%%%%%%%%%%%
%%%%%%%%%%%%%%%%%%%%%%%%%%%%%
%%%%%%%%%%%%%%%%%%%%%%%%%%%%%
%%%%%%%%%%%%%%%%%%%%%%%%%%%%%
%%%%%%%%%%%%%%%%%%%%%%%%%%%%%
%%%%%%%%%%%%%%%%%%%%%%%%%%%%%

\section{Asymptotic preserving (AP) and Asymptotically Accurate (AA) properties}
\label{sec4}
\setcounter{equation}{0}
\setcounter{figure}{0}
\setcounter{table}{0}
\subsection{AP property}
In this section, we will prove the AP property for the first order SI IMEX scheme \eqref{S3_A1}, and the AA property for the high order SI IMEX-RK scheme \eqref{a7}-\eqref{a10}. We focus on time discretizations while keeping the space continuous when we discuss the AP or AA property. First, we have the following theorem:
\begin{thm}
	The first order SI IMEX scheme \eqref{S3_A1} with space continuous is AP, in the sense that, with well-prepared initial conditions \eqref{S2_E18}, at the leading order asymptotic expansions, the scheme \eqref{S3_A1} is a consistent approximation of the incompressible MHD equations \eqref{S2_E26}.
\end{thm}

\begin{proof}
We assume the following expansions of the solutions at time $t^n$:
\begin{equation}
    \left\{
    \begin{aligned}
    &p^n(\mathbf{x}) = p^n_0 +  \eps^2 p^n_2(\mathbf{x}), \quad \rho^n(\mathbf{x}) = \rho^n_0(\mathbf{x}) + \mathcal{O}(\eps),\\
    &\mathbf{u}^n(\mathbf{x}) = \mathbf{u}^n_0(\mathbf{x}) + \mathcal{O}(\eps^2), \quad
    \mathbf{A}^n(\mathbf{x}) = \mathbf{A}^n_0(\mathbf{x}) + \mathcal{O}(\eps), \\
    &\mathbf{B}^n(\mathbf{x}) = \mathbf{B}^n_0(\mathbf{x}) + \mathcal{O}(\eps), &
    \end{aligned}\right.
    \label{4.1}
\end{equation}
with periodic or no-slip boundary conditions, and well-prepared conditions $p_0^n = (\gamma - 1)E^n_0 = {\rm Const}$, $\nabla \cdot \mathbf{u}^n_0(\mathbf{x}) = 0$ and $\nabla \cdot \bB^n(\bx) = 0$.
No matter what value $\bA^{n+1}$ takes, we can obtain $\nabla \cdot \bB^{n+1} = \nabla \cdot (\nabla \times \bA^{n+1}) = 0$ from \eqref{S3_E5}. We plug \eqref{4.1} into the semi-discrete scheme \eqref{S3_A1} and EOS \eqref{S3_E7}. For the leading order $\mathcal{O}(\eps^{-2})$, we obtain
\begin{equation*}
    \nabla p_0^{n+1} = 0,
\end{equation*}
i.e. $p_0^{n+1}$ is constant in space. Equating to zero the $\mathcal{O}(\eps^{0})$ terms, we have the following equations:
\begin{subequations}
    \label{4.3}
\begin{equation}
    \frac{\rho_0^{n+1} - \rho_0^n}{\Delta t} + \nabla \cdot (\rho_0^n \bu_0^n) = 0,
    \label{4.3a}
\end{equation}
\begin{equation}
    \frac{\rho_0^{n+1} \mathbf{u}_0^{n+1} - \rho_0^n \mathbf{u}_0^n}{\Delta t} + \nabla \cdot \left(\rho_0^n \bu_0^n \otimes \bu_0^n- \mathbf{B}_0^n \otimes \mathbf{B}_0^n +\frac{1}{2} {\Vert \mathbf{B}_0^n \Vert}^2 \mathbf{I} \right) +  \nabla p^{n+1}_2 = \mathbf{0},
    \label{4.3b}
\end{equation}
\begin{equation}
    \frac{\mathbf{A}_0^{n+1} - \mathbf{A}_0^n}{\Delta t} + (\nabla \times \bA_0^{n}) \times \bu_0^n = \mathbf{0},
    \label{4.3c}
\end{equation}
\begin{equation}
    \frac{E_0^{n+1} - E_0^n}{\Delta t} + \nabla \cdot \left(\bar{H}_0^n \left(\rho_0^{n+1} \bu_0^{n+1} \right) \right) = 0,
    \label{4.3d}
\end{equation}
\begin{equation}
    \bB_0^{n+1} - \nabla \times \bA_0^{n+1} = \mathbf{0},
    \label{4.3e}
\end{equation}
\begin{equation}
    E_0^{n+1} = \frac{p_0^{n+1}}{\gamma -1},
    \label{4.3f}
\end{equation}
\end{subequations}
with
\begin{equation}
    \bar{H}_0^n = \frac{p_0^n+ E_0^n}{\rho_0^{n+1}} = \frac{\gamma}{\gamma -1 } \frac{p_0^n}{\rho_0^{n+1}}.
    \label{4.4}
\end{equation}
Due to \eqref{4.3f} and \eqref{4.4}, it is easy to check that \eqref{4.3d} is equivalent to
\begin{equation}
    \frac{p_0^{n+1} - p_0^n}{\Delta t} + \gamma p_0^n \nabla \cdot \bu_0^{n+1} = 0.
    \label{4.5}
\end{equation}
 Taking into account a spatial domain $\Omega$ with periodic or no-slip boundary conditions and integrating \eqref{4.5} over $\Omega$, we further get $\int_{\Omega} \nabla \cdot \bu_0^{n+1} d \bx = \int_{\partial \Omega} \bu_0^{n+1} \cdot \bn d s = 0$, where $\bn$ is the unit outward normal vector along $\partial \Omega$. This implies $p_0^{n+1} = p^n_0 = {\rm Const}$. A direct result $\nabla \cdot \bu_0^{n+1} = 0$ can be obtained through combing \eqref{4.5} with $p_0^{n+1} = p^n_0$. The equations \eqref{4.3a}-\eqref{4.3c} and \eqref{4.3e} correspond to consistent discretizations for other equations in \eqref{S2_E26}, so that the AP property is obtained.
\end{proof}

\subsection{AA property}Similarly to the AP property, here we focus on the
AA analysis on time discretizations, while keeping space continuous. Then we have the following theorem:
\begin{thm}
	We consider a high order SI IMEX-RK scheme \eqref{a7}-\eqref{a10} of order r applied to system \eqref{S2_E24} on a bounded domain $\Omega \subset \mathbb{R}$ with periodic or compact support boundary conditions, and assume that the implicit part of the SI IMEX-RK scheme is SA and the initial conditions $(\rho^0(\bx),\rho^0(\bx) \bu^0(\bx), p^0,$ $ \bA^0(\bx), \bB^0(\bx))^T$ are well-prepared, namely in the form of \eqref{4.1}. If we denote $(\rho^1(\bx;\eps),\rho^1(\bx;\eps) \bu^1(\bx;\eps),$ $ p^1(\bx;\eps), \bA^1(\bx;\eps), \bB^1(\bx;\eps))^T$ as the numerical solution after one-time step, then, with $p_*$ being a constant, we have:
 \begin{equation}
    \label{4.6}
     \lim_{\eps \to 0} p^1(\bx;\eps) = p_*, \quad \lim_{\eps \to 0} \nabla \cdot \bu^1(\bx;\eps) = 0, \quad \nabla \cdot \bB^1(\bx;\eps) = 0.
 \end{equation}
Furthermore, let $\bV_{inc}(\bx;t) = (\rho_{inc}(\bx;t),\rho_{inc}(\bx;t) \bu_{inc}(\bx;t), p_{inc}(\bx;t), \bA_{inc}(\bx;t), \bB_{inc}(\bx;t))^T$ be the exact solution of the incompressible MHD equations \eqref{S2_E26} with the same initial data, one has the following one-step error estimate
\begin{equation}
    \lim_{\eps \to 0} \bV^1(\bx;\eps) = \bV_{inc}(\bx;\Delta t) + \mathcal{O}(\Delta t^{r+1}),
    \label{4.7}
\end{equation}
i.e. the scheme is AA.
\end{thm}
\begin{proof}
We consider the first step from $t^0 = 0$ to $t^1 = \Delta t$ for the SI IMEX-RK scheme \eqref{a7}-\eqref{a10} of order $r$ applied to system \eqref{S2_E24} with well-prepared initial data \eqref{4.1}:
\begin{equation*}
    \left\{
    \begin{aligned}
    &p^0(\mathbf{x}) = p_* +  \eps^2 p^0_2(\mathbf{x}), \quad \rho^0(\mathbf{x}) = \rho^0_{inc} + \mathcal{O}(\eps),\\
    &\mathbf{u}^0(\mathbf{x}) = \mathbf{u}^0_{inc} + \mathcal{O}(\eps^2), \quad
    \mathbf{A}^0(\mathbf{x}) = \mathbf{A}^0_{inc} + \mathcal{O}(\eps), \quad
    \mathbf{B}^0(\mathbf{x}) = \mathbf{B}^0_{inc} + \mathcal{O}(\eps), 
    \end{aligned}\right.
\end{equation*}
where
\begin{equation*}
    \begin{aligned}
    \rho^0_{inc} := \rho_{inc}(\bx,0),\quad \mathbf{u}^0_{inc} := \mathbf{u}_{inc}(\bx,0),\quad \mathbf{A}^0_{inc} := \mathbf{A}_{inc}(\bx,0),\quad \mathbf{B}^0_{inc} := \mathbf{B}_{inc}(\bx,0).
    \end{aligned}
\end{equation*}
By well-prepared assumptions we have: $p_{inc}(\bx,0) := p_*$ is a constant independent of time and space, $\nabla \cdot \bu_{inc}^0 = 0$ and $\nabla \cdot \bB^0(\bx) = 0$.  

Now we consider a formal $\eps$-expansion of the quantities:
\begin{equation*}
    U_I^{(i)} = (\rho_I^{(i)},\mathbf{q}_I^{(i)},\bA_I^{(i)},E_I^{(i)}), \quad U_E^{(i)} = (\rho_E^{(i)},\mathbf{q}_E^{(i)},\bA_E^{(i)},E_E^{(i)}),
\end{equation*} 
and $\bB_E^{(i)}$, with $\mathbf{q}_I^{(i)} = \rho_I^{(i)} \bu_I^{(i)}$ and $\mathbf{q}_E^{(i)} = \rho_E^{(i)} \bu_E^{(i)}$. For example, the expansions of the density and pressure are
\begin{equation}
    \left\{
    \begin{aligned}
    &\rho_I^{(i)} = \rho_{0,I}^{(i)} + \eps \rho_{1,I}^{(i)}+ \cdots, \quad &\rho_E^{(i)}& = \rho_{0,E}^{(i)} + \eps \rho_{1,E}^{(i)}+ \cdots,\\
    &p_I^{(i)} = p_{0,I}^{(i)} + \eps^2 p_{2,I}^{(i)}+ \cdots, \quad &p_E^{(i)}& = p_{0,E}^{(i)} + \eps^2 p_{2,E}^{(i)}+ \cdots .
    \end{aligned}\right.
    \label{g3}
\end{equation}
To prove the theorem, we use mathematical induction.
\bit
\item AA property for internal stages $i = 1, \cdots, s.$
When $i = 1$, it leads to the same AP analysis for the scheme \eqref{S3_A1} with $\Delta t$ replaced by $a_{11} \Delta t$. To prove the result for $i > 1$, we make use of the induction hypothesis,
assuming the property holds for $j \leq i-1$, and prove that it still holds for $j = i$. For $j = 1, \cdots , i- 1$, we have
\begin{equation}
    p_{0,E}^{(j)} = p_{0,I}^{(j)} = p_*, \quad E_{0,E}^{(j)} = E_{0,I}^{(j)} = \frac{p_*}{\gamma -1}, \quad \nabla \cdot \bu_{0,I}^{(j)} = 0, \quad \nabla \cdot \bB_E^{(j)} = 0.  
    \label{g4}
\end{equation}
Now we insert the expansions \eqref{g3} into the explicit step in \eqref{a7}, and from the energy equation we get 
\begin{equation}
    E_{0,E}^{(i)} = E_{inc}^0 - \Delta t \sum^{i-1}_{j = 1} \Tilde{a}_{ij} \nabla \cdot \left( \bar{H}_0^{(j)} \mathbf{q}_{0,I}^{(j)}\right),
    \label{g5}
\end{equation}
with $E_{inc}^0 = p_*/ (\gamma -1)$ and for $j=1, \cdots, i-1$,
\begin{equation}
    \bar{H}_0^{(j)} = \frac{E_{0,E}^{(j)} + p_{0,E}^{(j)}}{\rho_{0,I}^{(j)}} = \frac{\gamma}{\gamma -1} \frac{p_*}{\rho_{0,I}^{(j)}}.
    \label{g6}
\end{equation}
Now substituting \eqref{g4}, \eqref{g6}, and $E_{0,E}^{(i)} = p_{0,E}^{(i)} / (\gamma-1)$ into \eqref{g5}, we obtain
\begin{equation*}
    p_{0,E}^{(i)} = p_* - \Delta t \gamma p_* \sum^{i-1}_{j = 1} \Tilde{a}_{ij} \nabla \cdot \bu_{0,I}^{(j)} = p_*,
\end{equation*}
and $E_{0,E}^{(i)} = p_*/ (\gamma-1)$ is also a constant for the stage $i$. From \eqref{a7}, we have
\begin{equation}
\left\{
\begin{aligned}
    \label{416}
    &\mathbf{A}^{(i)}_E = \mathbf{A}^n - \Delta t \sum^{i-1}_{j = 1} \Tilde{a}_{ij} \left( \left(\nabla \times \bA^{(j)}_E \right) \times \mathbf{u}^{(j)}_E \right),\\
    &\bB_E^{(i)} = \nabla \times \bA_E^{(i)}.  
\end{aligned}
\right.
\end{equation}
From \eqref{416}, we get $\nabla \cdot \bB_E^{(i)} = \nabla \cdot \left( \nabla \times \bA_E^{(i)}\right) = 0$. From \eqref{a7}, in the order of $\mathcal{O}(1)$, the density and momentum equations yield
\begin{equation}
\left\{
\begin{aligned}
    &\rho_{0,E}^{(i)} = \rho_{inc}^{0} - \Delta t \sum^{i-1}_{j = 1} \Tilde{a}_{ij} \nabla \cdot \mathbf{q}_{0,E}^{(j)},\\
    &\mathbf{q}^{(i)}_{0,E} = \mathbf{q}^0_{inc} - \Delta t \sum^{i-1}_{j = 1} \Tilde{a}_{ij} \left(\nabla \cdot \left(\rho_{0,E}^{(j)} \mathbf{u}_{0,E}^{(j)} \otimes \mathbf{u}_{0,E}^{(j)} - \mathbf{B}^{(j)}_{0,E} \otimes \mathbf{B}^{(j)}_{0,E} +\left(\frac{1}{2} {\Vert \mathbf{B}^{(j)}_{0,E} \Vert}^2 \right) \mathbf{I} \right) + \nabla p^{(j)}_{2,I} \right),
    \label{g16}
\end{aligned}
\right.
\end{equation}
with $\mathbf{q}^0_{inc} = (\rho \bu)^0_{inc}$ and $\nabla p_{0,E}^{(j)} = 0$ for $j=1,\cdots,i-1$.\\
Similarly, inserting expansions \eqref{g3} into \eqref{a8}, in the order of $\mathcal{O}(1)$, we get
\begin{equation}
\left\{
\begin{aligned}
    &\rho_{0,*}^{(i)} = \rho_{inc}^{0} - \Delta t \sum^{i-1}_{j = 1} {a}_{ij} \nabla \cdot \mathbf{q}_{0,E}^{(j)},\\
    &\mathbf{q}^{(i)}_{0,*} = \mathbf{q}^0_{inc} - \Delta t \sum^{i-1}_{j = 1} {a}_{ij} \left(\nabla \cdot \left(\rho_{0,E}^{(j)} \mathbf{u}_{0,E}^{(j)} \otimes \mathbf{u}_{0,E}^{(j)} - \mathbf{B}^{(j)}_{0,E} \otimes \mathbf{B}^{(j)}_{0,E} +\left(\frac{1}{2} {\Vert \mathbf{B}^{(j)}_{0,E} \Vert}^2 \right) \mathbf{I} \right) + \nabla p^{(j)}_{2,I} \right),  
    \label{g17}
\end{aligned}
\right.
\end{equation}
where from \eqref{p12} and \eqref{g4}, it follows  $\nabla \bar{p}^{(j)}_E = 0$ for $j = 1, \cdots , i-1$. Using \eqref{g6}, we get
\begin{equation}
    E_{0,*}^{(i)} = E_{inc}^0 - \Delta t \sum^{i-1}_{j = 1} {a}_{ij} \nabla \cdot \left( \bar{H}_0^{(j)} \mathbf{q}_{0,I}^{(j)}\right) = \frac{p_*}{\gamma-1}- \frac{\gamma p_*}{\gamma-1} \Delta t \sum^{i-1}_{j = 1} {a}_{ij} \nabla \cdot \bu_{0,I}^{(j)}.
    \label{g18}
\end{equation}
Therefore, from \eqref{g4} and \eqref{g18}, we get
\begin{equation*}
    E_{0,*}^{(i)} = \frac{p_*}{\gamma-1}.
\end{equation*}
Now from \eqref{g18} and \eqref{a9}, it follows for the energy equation
\begin{equation*}
    E^{(i)}_{0,I} = E^{(i)}_{0,*} - \Delta t a_{ii} \nabla \cdot (\bar{H}^{(i)}_{0} \mathbf{q}^{(i)}_{0,I}) = \frac{p_*}{\gamma-1}- \frac{\gamma p_*}{\gamma-1} \Delta t {a}_{ii} \nabla \cdot \bu_{0,I}^{(i)}
\end{equation*}
Considering the EOS $E_I^{(i)} = p_I^{(i)} / (\gamma -1) + \eps^2 ({\Vert\mathbf{q}_E^{(i)} \Vert}^2 / (2 \rho_E^{(i)}) + {\Vert\mathbf{B}_E^{(i)} \Vert}^2/2)$, to zeroth order in $\eps$, we get $E_{0,I}^{(i)} = p_{0,I}^{(i)}/(\gamma -1)$, and we obtain for the pressure
\begin{equation*}
    p^{(i)}_{0,I} = p_* - \gamma p_* \Delta t {a}_{ii} \nabla \cdot \bu_{0,I}^{(i)}.
\end{equation*}
Integrating it over the spatial bounded domain $\Omega$, and assuming periodic or compact support boundary conditions, we first obtain $p^{(i)}_{0,I} = p_*$, and then we get $\nabla \cdot \bu_{0,I}^{(i)} = 0$ at the stage $i$.

Finally, from \eqref{a9}, considering \eqref{g17}, we get for the density and momentum:
\begin{equation}
\left\{
\begin{aligned}
    &\rho_{0,I}^{(i)} = \rho_{inc}^{0} - \Delta t \sum^{i}_{j = 1} {a}_{ij} \nabla \cdot \mathbf{q}_{0,E}^{(j)},\\
    &\mathbf{q}^{(i)}_{0,I} = \mathbf{q}^0_{inc} - \Delta t \sum^{i}_{j = 1} {a}_{ij} \left(\nabla \cdot \left(\rho_{0,E}^{(j)} \mathbf{u}_{0,E}^{(j)} \otimes \mathbf{u}_{0,E}^{(j)} - \mathbf{B}^{(j)}_{0,E} \otimes \mathbf{B}^{(j)}_{0,E} +\left(\frac{1}{2} {\Vert \mathbf{B}^{(j)}_{0,E} \Vert}^2 \right) \mathbf{I} \right) + \nabla p^{(j)}_{2,I} \right).  
    \label{g22}
\end{aligned}
\right.
\end{equation}
%where from \eqref{p12} and $p_{0,E}^{(i)} = p_*$, it follows in the equation of the momentum $\nabla \bar{p}^{(i)}_E = 0$ for $i$.
Then equations \eqref{416}, \eqref{g16}, and \eqref{g22}, with $\nabla \cdot \bB_E^{(i)} = 0$, $p_{0,E}^{(i)} = p_{0,I}^{(i)} = p_*$, and the divergence-free leading order velocity, i.e. $\nabla \cdot \bu_{0,I}^{(i)} = 0$, provide a multi-stage SI IMEX-RK discretization of the limiting system \eqref{S2_E26} for the internal stage $i$. %This shows that in the limit $\eps \rightarrow 0$, the scheme becomes the same SI-IMEX-RK time discrete scheme for the incompressible MHD equations \eqref{S2_E26}.
\eit

\bit
\item AA property for the updated numerical solution.
Assuming that SI IMEX-RK scheme in Section \ref{DBT} is SA, then the numerical solution coincides with the last internal stage $s$, and by setting $i = s$, we get
\begin{equation}
    p_0^1 = p_{0,I}^{(s)} = p_*, \quad \nabla \cdot \bu_{0}^{1} = \nabla \cdot \bu_{0,I}^{(s)} = 0,
    \label{23}
\end{equation}
and for the magnetic field
\begin{equation}
    \nabla \cdot \bB_E^{1} =\nabla \cdot ( \nabla \times \bA_E^{(1)}) = 0,  
    \label{24}
\end{equation}
namely we have \eqref{4.6}.
Now if we denote
\begin{equation*}
    \bV_{inc}(\bx;t) = (\rho_{inc}(\bx;t),\rho_{inc}(\bx;t) \bu_{inc}(\bx;t), p_{inc}(\bx;t), \bA_{inc}(\bx;t), \bB_{inc}(\bx;t))^T
\end{equation*}
as the exact solution of the incompressible MHD equations \eqref{S2_E26}, with initial data $\bV_{inc}(\bx;0) = (\rho^0(\bx),\rho^0(\bx) \bu^0(\bx), p^0, \bA^0(\bx), \bB^0(\bx))^T$, from equations \eqref{416}, \eqref{g16}, \eqref{g22}, \eqref{23}, and \eqref{24}, one gets in the low sonic Mach limit where $\eps = 0$, a high order SI IMEX-RK scheme of order $r$ for the numerical solutions of \eqref{S2_E26}, that is, the scheme \eqref{a7}-\eqref{a10} is AA, and \eqref{4.7} holds.
\eit
\end{proof}

%%%%%%%%%%%%%%%%%%%%%%%%%%%%%
%%%%%%%%%%%%%%%%%%%%%%%%%%%%%
%%%%%%%%%%%%%%%%%%%%%%%%%%%%%
%%%%%%%%%%%%%%%%%%%%%%%%%%%%%
%%%%%%%%%%%%%%%%%%%%%%%%%%%%%
%%%%%%%%%%%%%%%%%%%%%%%%%%%%%

\section{Numerical tests}
\label{sec5}
\setcounter{equation}{0}
\setcounter{figure}{0}
\setcounter{table}{0}

In this section, we perform some numerical experiments to validate the high order accuracy, AP, AA, divergence-free, and good performances of our proposed scheme for the MHD system with all sonic Mach numbers. The third order SI IMEX-RK with an SA property from \cite{boscarino2022high,huang2022high} is used for time discretizations. The fifth order finite difference WENO reconstruction\cite{boscarino2022high,huang2022high,shu1998essentially,jiang2000weighted} for the first order spatial derivatives, and a fourth order compact central difference discretization \cite{boscarino2022high,boscarino2019high} for the second order derivatives will be applied for spatial discretizations. Our scheme is referred as the ``IMEX" scheme. For some of the following examples, we will compare our results to reference solutions, which are produced by a fifth-order finite difference WENO scheme with a third-order strong stability preserving RK method. It is referred as an ``ERK” scheme if without a CT method \cite{jiang1999high}, while as an ``ERKC" scheme if with a CT method \cite{christlieb2014finite}. To avoid excessive numerical dissipations for low sonic Mach numbers, for the explicit schemes ``ERK'' and ``ERKC'', here we also take a Lax-Friedrichs flux the same as in \eqref{LF} and \eqref{3.21}, and numerically we find that it leads to much better results as we can see in the following. Unless otherwise specified, the ${\rm CFL}$ number is taken as $0.25$, and the gas constant is $\gamma = 5/3$. For 1D problems, the time step is
\begin{equation*}
    \Delta t = \frac{{\rm CFL} \Delta x}{\mathop{\max}\limits_{0 \leq k \leq N } \left(\vert u_k \vert + \hat{c}_f^{(k)}\right)},
\end{equation*}
with $N+1$ computational grid points, and the definition of $\hat{c}_f^{(k)}$ has been given in \eqref{3.21}. For 2D problems, the time step is
\begin{equation*}
    \Delta t = {\rm CFL} / \left(\frac{\mathop{\max}\limits_{0 \leq k \leq N_x, 0 \leq l \leq N_y } \left(\vert u_{k,l} \vert + \hat{c}_{f,x}^{(k,l)}\right)}{\Delta x} + \frac{\mathop{\max}\limits_{0 \leq k \leq N_x, 0 \leq l \leq N_y } \left(\vert v_{k,l} \vert + \hat{c}_{f,y}^{(k,l)}\right)}{\Delta y}\right),
    \label{s52}
\end{equation*}
with $(N_x+1) \times (N_y+1)$ computational grid points. Here, the fast speeds in the $x$-direciton ($\hat{c}_{f,x}^{(k,l)}$) and in the $y$-direciton ($\hat{c}_{f,y}^{(k,l)}$) are defined as
\begin{equation*}
\left\{
\begin{aligned}
    \hat{c}_{f,x}^{(k,l)} = {\left\{ \frac{1}{2} \left[ \left(\hat{a}^{(k,l)} \right)^2 + \frac{{\Vert \bB_{k,l} \Vert}^2}{\rho_{k,l} } \pm \sqrt{{\left(\left(\hat{a}^{({k,l} )} \right)^2 + \frac{{\Vert \bB_{k,l}  \Vert}^2}{\rho_{k,l} }\right)}^2 - \left(2{\hat{a}^{({k,l} )}} \sqrt{\frac{B_{x;{k,l}}^2}{\rho_{k,l} }
    } \right)^2 } \right] \right\}}^{\frac{1}{2}},\\
    \hat{c}_{f,y}^{(k,l)} = {\left\{ \frac{1}{2} \left[ \left(\hat{a}^{(k,l)} \right)^2 + \frac{{\Vert \bB_{k,l} \Vert}^2}{\rho_{k,l} } \pm \sqrt{{\left(\left(\hat{a}^{({k,l} )} \right)^2 + \frac{{\Vert \bB_{k,l}  \Vert}^2}{\rho_{k,l} }\right)}^2 - \left(2{\hat{a}^{({k,l} )}} \sqrt{\frac{B_{y;{k,l}}^2}{\rho_{k,l} }
    } \right)^2 } \right] \right\}}^{\frac{1}{2}},
\end{aligned}
\right.
\end{equation*}
where
\begin{equation*}
    \hat{a}^{({k,l})} = \min\left(\frac{1}{\eps},1 \right) \sqrt{\frac{\gamma p_{k,l}}{\rho_{k,l}}}.
\end{equation*}
\subsection{One dimensional case}
 Since the divergence-free condition is satisfied automatically in 1D, therefore, the CT method will not need to be applied in the following 1D numerical simulations.
\begin{exa} {\em
\label{exam1}
({\bf{Accuracy test}})
To assess the convergence of our proposed scheme, we simulate the nonlinear circularly polarized Alfvén wave problem in 1D. The smooth initial conditions are given as \cite{christlieb2016high}:
\begin{equation}
\label{5.1}
\begin{aligned}
       & (\rho,u,v,w,B_x,B_y,B_z,p)(0,x) \\
    &\quad = (1, 0, 0.1 \sin(2 \pi x), 0.1 \cos(2 \pi x), 1, 0.1 \sin(2 \pi x), 0.1 \cos(2 \pi x), 0.1) \quad x \in [0,1],
\end{aligned}
\end{equation}
with a periodic boundary condition. The exact solution to \eqref{5.1} propagates with a unit Alfvén speed  (that is $U(t,x) = U(0,x+t)$.). We set the time step to be
\begin{equation*}
    \Delta t = \frac{{\rm CFL} \Delta x^\frac{5}{3}}{\mathop{\max}\limits_{0 \leq k \leq N } \left(\vert u_k \vert + \hat{c}_f^{(k)}\right)}.
\end{equation*}
The sonic Mach number is $\eps = 1$ and we compute up to a final time $T = 1$. The numerical $L_{1}$, $L_{2}$, and $L_{\infty}$ errors and orders are shown in Table \ref{T1}. We can observe that a fifth-order spatial accuracy is obtained.

\begin{table}[htbp]
  \caption{ Example \ref{exam1}. The  $L_{1}$, $L_{2}$, and $L_{\infty}$ errors and orders for $\rho v$ with $\eps=1 $  at $T = 1$. }
  \begin{center}
  \begin{tabular}{|c|c|c|c|c|c|}\hline
   $N$ & 10 & 20 & 40 & 80 & 160   \\ \hline
   {$L_{1}$ error} & 6.24E-04 & 2.03E-05 & 6.41E-07 & 2.01E-08 & 6.27E-10\\ \hline
   order & -- & 4.94 & 4.99 & 5.00 & 5.00 \\ \hline
   {$L_{2}$ error} & 6.83E-04 & 2.25E-05 & 7.11E-07 & 2.23E-08 & 6.97E-10\\ \hline
   order & -- & 4.93 & 4.98 & 5.00 & 5.00  \\ \hline
   {$L_{\infty}$ error} & 9.65E-04 & 3.15E-05 & 1.00E-06 & 3.15E-08 & 9.85E-10\\ \hline
   order & -- & 4.94 & 4.97 & 4.99 & 5.00  \\ \hline
  \end{tabular}
  \end{center}
  \label{T1}
\end{table}
}
\end{exa}

\begin{exa} {\em
\label{exam2}
({\bf{Accuracy test for a range of $\eps$}}) We take a well-prepared initial data which is $\eps-$dependent:
\begin{equation*}
    \left\{
    \begin{aligned}
        \rho(x,0) & = 1+\eps^2\sin^2(2 \pi x), \quad &p(x,0) & = (1+\eps^2\sin^2(2 \pi x))^{\gamma}, \\
        u(x,0) & = \eps^2\sin(2 \pi x), \quad &B_x(x,0) & = 0.5, \\
        v(x,0) & = \sin(2 \pi x)+\eps^2 \cos(2 \pi x), \quad &B_y(x,0) & = (1 + \eps^2) \sin(2 \pi x), \\
        w(x,0) & = 0, \quad &B_z(x,0) & = (1 + \eps^2) \cos(2 \pi x), \\     
    \end{aligned}
    \right.
\end{equation*}
on the domain $\Omega = [0,1]$ with $\gamma = 1.4$ and periodic boundary conditions. Reference solutions are computed with $N = 320$. Four different sonic Mach numbers $\eps = 1, 10^{-2}, 10^{-6}$, and $0$ are taken, with a final time $T = 0.05$. Numerical errors and orders are shown in
Table \ref{T2}. For the intermediate value of $\eps = 10^{-2}$, order reduction
is observed, which is a typical behavior of high-order IMEX schemes for these multi-scale
problems \cite{boscarino2022high}. Around 5th-order accuracy for $\eps = 1, 10^{-6}$, and $0$ is observed.
\begin{table}[htbp]
  \caption{ Example \ref{exam2}. The  $L_{1}$ errors and orders for $\rho v$ with $\eps=1,  10^{-2}, 10^{-6}$, and $0$. }
  \begin{center}
		\begin{tabular}{|c|c|c|c|c|c|c|c|c|}
			\hline
			\multicolumn{1}{|c|}{\multirow{2}*{\diagbox{$N$}{$\eps$}}}&\multicolumn{2}{c|}{ $\eps=1$}&\multicolumn{2}{c|}{$\eps=10^{-2}$} &\multicolumn{2}{c|}{$\eps=10^{-6}$}&\multicolumn{2}{c|}{$\eps=0$}\\
			\cline{2-9}
\multicolumn{1}{|c|}{} &$L_1$ error& order &$L_1$ error& order&$L_1$ error& order&$L_1$ error& order\\  \cline{1-9}
10 & 3.08E-02 & --   & 4.93E-04 & --   & 4.82E-04 & --   & 4.82E-04 & --\\ \hline
20 & 6.47E-03 & 2.25 & 7.18E-05 & 2.78 & 1.61E-05 & 4.90 & 1.61E-05 & 4.90\\ \hline
40 & 8.32E-04 & 2.96 & 3.33E-05 & 1.11 & 5.22E-07 & 4.95 & 5.22E-07 & 4.95\\ \hline
80 & 3.74E-05 & 4.48 & 1.01E-05 & 1.72 & 1.76E-08 & 4.89 & 1.76E-08 & 4.89\\ \hline
160 & 1.22E-06 & 4.94& 3.34E-06 & 1.60 & 6.65E-10 & 4.72 & 6.65E-10 & 4.72\\ \hline
\end{tabular}
	\end{center}
  \label{T2}
\end{table}

}
\end{exa}

\begin{exa} {\em
\label{exam3}
({\bf{MHD shock tube}}) In this example, we consider a 1D MHD shock tube in the compressible regime where the sonic Mach number is of $\mathcal{O}(1)$. The initial conditions are as follows
\begin{equation*}
\begin{aligned}
        &(\rho,u,v,w,B_x,B_y,B_z,p)(0,x) \\
    &\quad = \left\{
    \begin{aligned}
        &(1, 0, 0, 0, 0.75, 1, 0, 1) \quad& x &\in [0,0.5),\\
        &(0.125, 0, 0, 0, 0.75, -1, 0, 0.1) \quad& x& \in [0.5,1],
    \end{aligned}
    \right.
\end{aligned}
\end{equation*}
with a final time $T = 0.1$. The specific heat ratio is set to be $\gamma = 2$ with $\eps = 1$ on the domain $\Omega = [0,1]$. Reflective boundary conditions are considered, and we take $N = 200$. The results are shown in Fig.~\ref{F1}, and the solutions obtained by the ERK scheme with $N=2400$ are displayed as references. From these results, it can be seen that our IMEX scheme can capture the MHD waves very well.
\begin{figure}[hbtp]
\begin{center}
	\mbox{
		{\includegraphics[width=8cm]{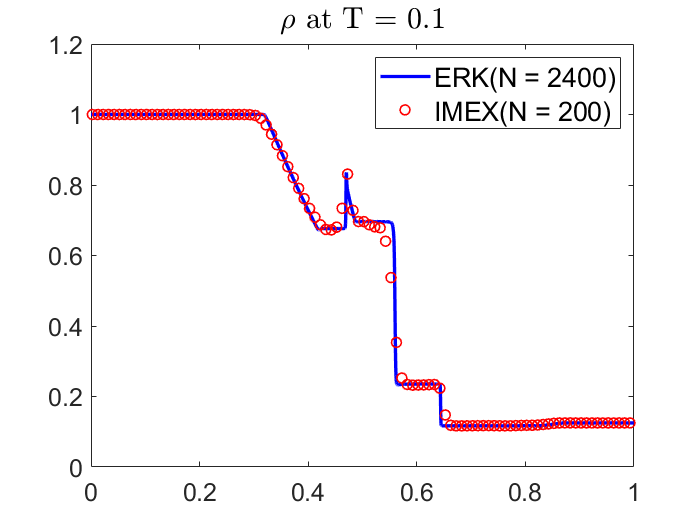}}\quad
		{\includegraphics[width=8cm]{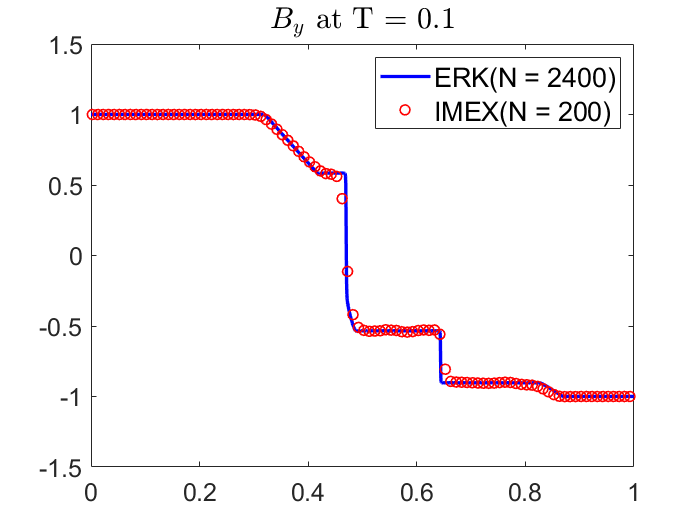}}
		}
\caption{ MHD shock tube solutions for Example \ref{exam3}. Left: density $\rho$; Right: magnetic field $B_y$. Every two points are plotted for the IMEX scheme.}
\label{F1}
\end{center}
\end{figure}

}
\end{exa}

\subsection{Two dimensional case}
A CT method described in \cite{christlieb2014finite} will be used to maintain the divergence-free condition. In 2D, $\bB = \nabla \times \bA$ can be written as $B_x = \partial A_z / \partial y$ and $B_y = -\partial A_z / \partial x$. Therefore, we only need to update the third component $A_z$ of $\bA$, and ignore $A_x$ and $A_y$. All solutions of the following tests maintain $\nabla \cdot \bB = 0$ up to machine round off errors\cite{christlieb2016high}.
\begin{exa} {\em
\label{exam4}
({\bf{Accuracy test}}) The 2D version of the smooth Alfvén wave problem is obtained by rotating the direction of propagation by an angle of $\theta$, so that the wave now propagates along the direction $\bn = (-\cos{\theta}, -\sin{\theta})$ on the domain $\Omega = [0,1/\cos{\theta}] \times [0,1/\sin{\theta}]$ \cite{christlieb2016high}. In this numerical test, periodic boundary conditions are applied on all both directions with $\eps = 1, \theta = \pi /4$, and $T = 1$. Similarly, if we take 
\begin{equation*}
     \Delta t = {\rm CFL} / \left(\frac{\mathop{\max}\limits_{0 \leq k \leq N_x, 0 \leq l \leq N_y } \left(\vert u_{k,l} \vert + \hat{c}_{f,x}^{(k,l)}\right)}{\Delta x^\frac{5}{3}} + \frac{\mathop{\max}\limits_{0 \leq k \leq N_x, 0 \leq l \leq N_y } \left(\vert v_{k,l} \vert + \hat{c}_{f,y}^{(k,l)}\right)}{\Delta y^\frac{5}{3}}\right),
\end{equation*}
from Table \ref{T3} we can find around 5th-order accuracy for the momentum $\rho u$, and 4th-order accuracy for the magnetic field component $B_x$ where a fourth-order central difference discretization is used for $\bB = \nabla \times \bA$.

\begin{table}[htbp]
  \caption{ Example \ref{exam4}. The $L_{1}$, $L_{2}$, and $L_{\infty}$ errors and orders for $\rho u$ and $B_x$ with $\eps=1$  and $T = 1$. }
  \begin{center}
  \begin{tabular}{|c|c|c|c|c|c|c|}\hline
    &$N_x \times N_y$ & $8 \times 8$ & $16 \times 16$ & $32 \times 32$ & $64 \times 64$ & $128 \times 128$  \\ \hline
   {\multirow{6}*{$\rho u$}}
   &{$L_{1}$ error} & 1.54E-03 & 5.73E-05 & 1.86E-06 & 5.88E-08 & 1.84E-09 \\ \cline{2-7}
   &order & -- & 4.75 & 4.94 & 4.99 & 5.00 \\ \cline{2-7}
   &{$L_{2}$ error} & 1.80E-03 & 6.47E-05 & 2.08E-06 & 6.55E-08 & 2.07E-09 \\ \cline{2-7}
   &order & -- & 4.80 & 4.96 & 4.99 & 4.99 \\ \cline{2-7}
   &{$L_{\infty}$ error} & 2.55E-03 & 9.26E-05 & 2.98E-06 & 9.44E-08 & 3.00E-09 \\ \cline{2-7}
   &order & -- & 4.79 & 4.96 & 4.98 & 4.98 \\ \hline
   {\multirow{6}*{$B_x$}}
   &{$L_{1}$ error} & 9.29E-04 & 6.36E-05 & 3.17E-06 & 1.69E-07 & 9.64E-09 \\ \cline{2-7}
   &order & -- & 3.87 & 4.33 & 4.23 & 4.13 \\ \cline{2-7}
   &{$L_{2}$ error} & 1.08E-03 & 7.02E-05 & 3.51E-06 & 1.88E-07 & 1.07E-08 \\ \cline{2-7}
   &order & -- & 3.94 & 4.32 & 4.23 & 4.13 \\ \cline{2-7}
   &{$L_{\infty}$ error} & 1.29E-03 & 9.69E-05 & 5.06E-06 & 2.70E-07 & 1.52E-08 \\ \cline{2-7}
   &order & -- & 3.73 & 4.26 & 4.23 & 4.15 \\ \hline
  \end{tabular}
  \end{center}
  \label{T3}
\end{table}
}
\end{exa}

\begin{exa} {\em
\label{exam5}
({\bf{Accuracy test for a range of $\eps$}}) In this 2D case, we set a well-prepared initial condition as follows
\begin{equation*}
    \left\{
    \begin{aligned}
        \rho(x,y,0) & = 1+\eps^2\sin^2(2 \pi (x+y)), \quad &p(x,y,0) & = (1+\eps^2\sin^2(2 \pi (x+y)))^{\gamma}, \\
        u(x,y,0) & = \sin(2 \pi (x-y)) + \eps^2 \sin(2 \pi (x+y)), \quad &B_x(x,y,0) & = -\frac{1}{\sqrt{2}}\sin(2 \pi (x+y)), \\
        v(x,y,0) & = \sin(2 \pi (x-y)) + \eps^2 \cos(2 \pi (x+y)), \quad &B_y(x,y,0) & = \frac{1}{\sqrt{2}}\sin(2 \pi (x+y)), \\
        w(x,y,0) & = 0, \quad &B_z(x,y,0) & = \cos(2 \pi (x+y)), \\     
    \end{aligned}
    \right.
\end{equation*}
with the initial magnetic potential $A_z(x,y,0) = \cos(2 \pi (x+y)) / (2 \sqrt{2} \pi)$. As in Example \ref{exam2}, four different sonic Mach numbers are considered with periodic boundary conditions on the domain $\Omega = [0,1] \times [0,1]$. We compute the solution up to a final time
$T = 0.01$ on mesh grid points of $N^2$. The $L_1$ errors and orders of the accuracy are shown in Table \ref{T4}. From this Table, we can see high-order accuracy can be obtained for $\eps = 1, 10^{-6}$, and $0$. However, under the current mesh sizes, order degeneracy can also be found for $\eps = 10^{-2}$, which is similar to the 1D case.
\begin{table}[htbp]
  \caption{ Example \ref{exam5}. The  $L_{1}$ errors and orders for $\rho u$ with $\eps=1,  10^{-2}, 10^{-6}$, and $0$. }
  \begin{center}
		\begin{tabular}{|c|c|c|c|c|c|c|c|c|}
			\hline
			\multicolumn{1}{|c|}{\multirow{2}*{\diagbox{$N$}{$\eps$}}}&\multicolumn{2}{c|}{ $\eps=1$}&\multicolumn{2}{c|}{$\eps=10^{-2}$} &\multicolumn{2}{c|}{$\eps=10^{-6}$}&\multicolumn{2}{c|}{$\eps=0$}\\
			\cline{2-9}
\multicolumn{1}{|c|}{} &$L_1$ error& order &$L_1$ error& order&$L_1$ error& order&$L_1$ error& order\\  \cline{1-9}
8 & 8.98E-02 & -- & 5.18E-03 & -- & 7.65E-03 & -- & 7.65E-03 & --\\ \hline
16 & 1.39E-02 & 2.69 & 2.09E-03 & 1.31 & 8.74E-04 & 3.13 & 8.74E-04 & 3.13\\ \hline
32 & 6.76E-04 & 4.36 & 2.09E-03 & -- & 1.68E-05 & 5.70 & 1.68E-05 & 5.70\\ \hline
64 & 2.28E-05 & 4.89 & 2.41E-03 & -- & 7.42E-07 & 4.50 & 7.43E-07 & 4.50\\ \hline
128 & 7.15E-07 & 5.00 & 6.41E-04 & 1.91 & 1.42E-08 & 5.70 & 1.40E-08 & 5.73\\ \hline
\end{tabular}
	\end{center}
  \label{T4}
\end{table}
}
\end{exa}

\begin{exa} {\em
\label{exam6}
({\bf{Orszag-Tang vortex}}) Next we consider the Orszag–Tang vortex problem, which is widely considered as a standard test for MHD \cite{christlieb2014finite,christlieb2016high,tang2000high,balsara2014multidimensional}.
The problem has smooth initial conditions
\begin{equation*}
    \begin{aligned}
        &(\rho,u,v,w,B_x,B_y,B_z,p)(x,y,0) \\
    =& (\gamma^2, -\sin{(y)}, \sin{(x)}, 0, -\sin{(y)}, \sin{(2x)}, 0, \gamma),
\end{aligned}
\end{equation*}
with the initial magnetic potential: $A_z = 0.5 \cos{(2x)} + \cos{(y)}$ on the domain $\Omega = [0,2\pi] \times [0,2\pi]$. Periodic boundary conditions are imposed on all boundaries. As time evolves, the solution forms several shock waves and a vortex structure in the middle of the computational domain. We set $\eps = 1$, $N_x = N_y = 192$ and present the density at $T = 0.5, T = 2, T = 3$, and $T = 4$ in Fig.~\ref{F6a}. A slice of the pressure at $y = 0.625 \pi$ and $T = 3$ is shown on the right panel of Fig.~\ref{F6b}. We find our IMEX scheme can successfully capture the shocks. We do not observe significant oscillations in any of the conserved quantities, and our results are in good agreement with those given in \cite{christlieb2014finite,christlieb2016high}.
\begin{figure}[hbtp]
\begin{center}
	\mbox{
		{\includegraphics[width=7cm]{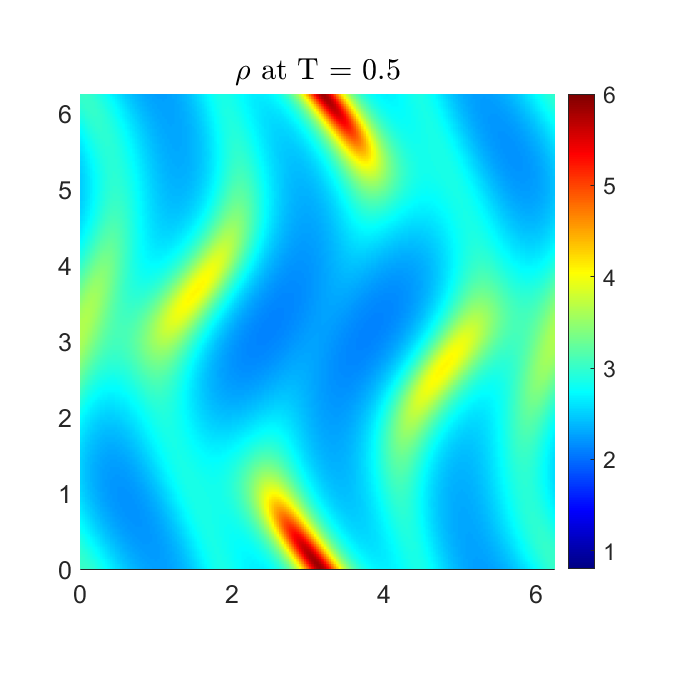}}
		{\includegraphics[width=7cm]{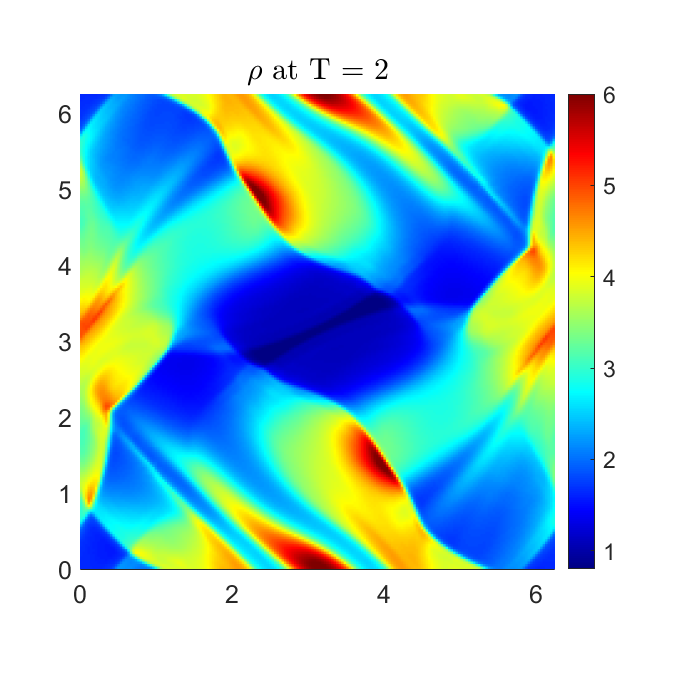}}
  }\\ \mbox{
        {\includegraphics[width=7cm]{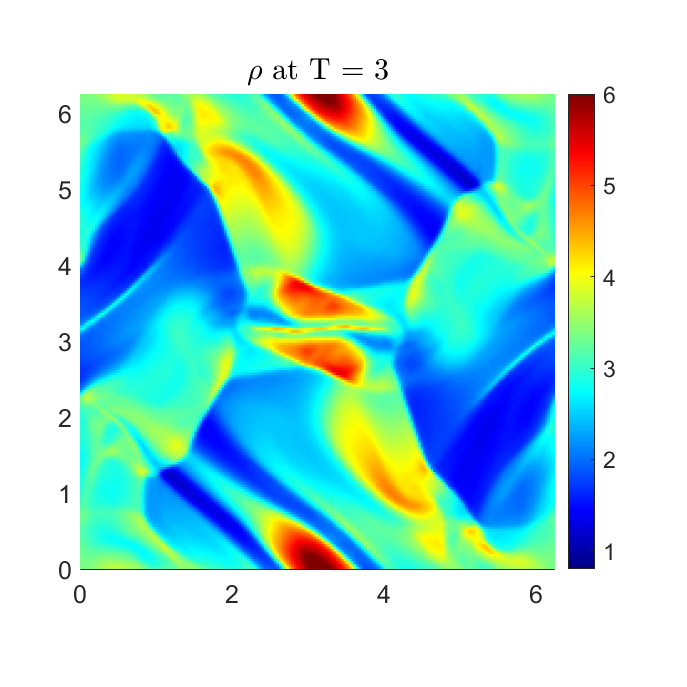}}
		{\includegraphics[width=7cm]{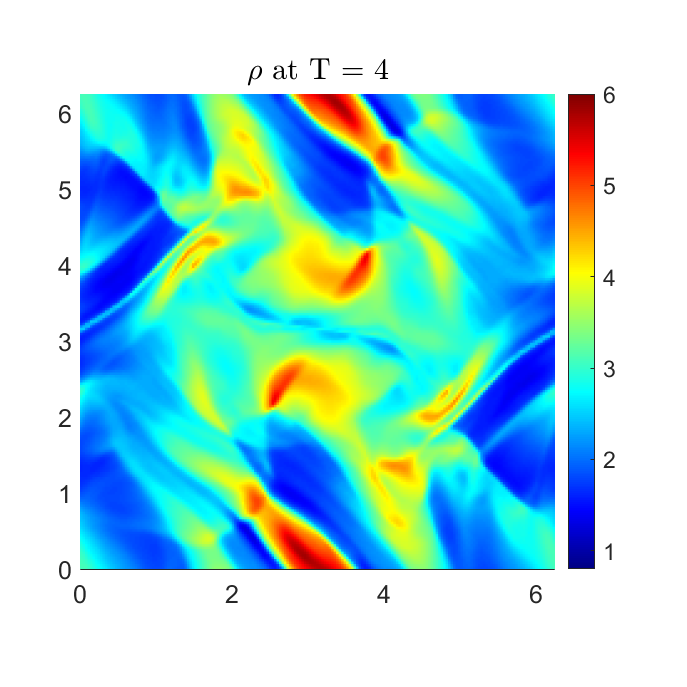}}
		}
\caption{ Density for Orszag-Tang problem of Example \ref{exam6} on the mesh grid $192^2$.}
\label{F6a}

\end{center}
\end{figure}
\begin{figure}[htbp]
\begin{center}
	\mbox{
        {\includegraphics[width=7cm]{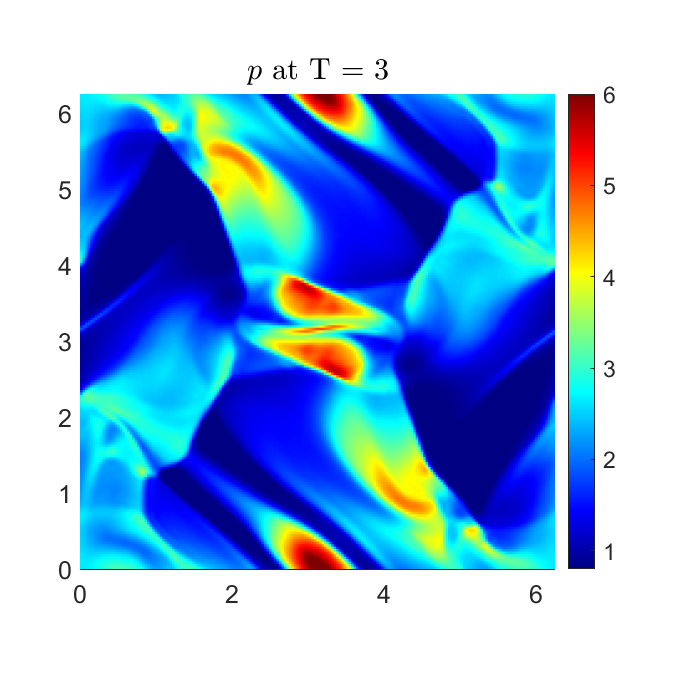}}
        \raisebox{0.04\height}
        {\includegraphics[width=7cm]{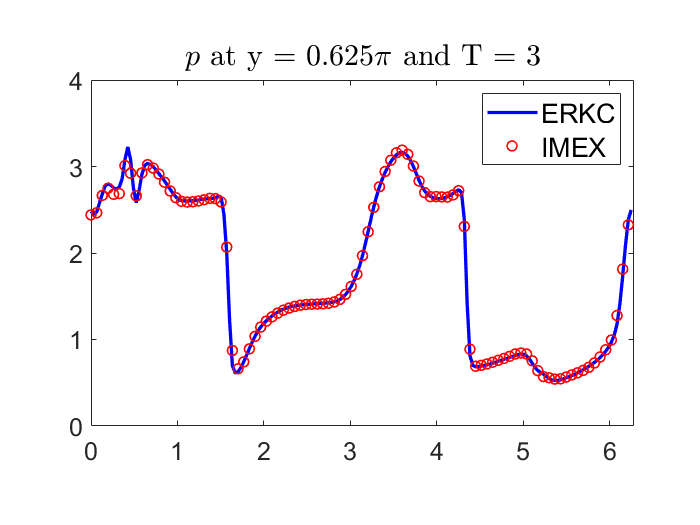}}
        }
    % \subfigure[]{\includegraphics[width=7cm]{./pic/2D/Test6/p3.png}}
    % \subfigure[]{\includegraphics[width=7cm]{./pic/2D/Test6/p3slice.png}}
\caption{ Example \ref{exam6}. Thermal Pressure for Orszag-Tang problem on the mesh grid $192^2$. Every two points are plotted for the IMEX scheme for the cutting plot.}
\label{F6b}
\end{center}
\end{figure}
}
\end{exa}

\begin{exa} {\em
\label{exam7}
({\bf{MHD blast wave}}) MHD blast wave has been commonly
used to test numerical methods for the MHD system \cite{christlieb2018high,han2007adaptive}. Here, we set periodic boundary conditions on the domain $\Omega = [-0.5,0.5] \times [-0.5,0.5]$. The initial conditions are
\begin{equation*}
\begin{aligned}
        &(\rho,u,v,w,B_x,B_y,B_z,p)(x,y,0) \\
   &\quad = \left\{
    \begin{aligned}
        &(1, 0, 0, 0, 10\sin{\theta}, 10\cos{\theta
        }, 0, 100) \quad& r &\leq 0.125,\\
        &(1, 0, 0, 0, 10\sin{\theta}, 10\cos{\theta
        }, 0, 10) \quad& r &> 0.125,
    \end{aligned}
    \right.
\end{aligned}
\end{equation*}
with $\theta = \pi / 4$ and $r = \sqrt{x^2+y^2}$. The initial magnetic potential is given by $A_z(x,y,0) = 5 \sqrt{2}(-x+y)$ with a first order extrapolation if outside the boundary. In Fig.~\ref{F7a}, we show the density and magnetic pressure obtained by the IMEX and ERKC schemes on a $200 \times 200$ mesh with $\eps = 0.9$ at $T=0.02$. The slices of the density and the magnetic
pressure along $y = -x+0.5$ are also shown in Fig.~\ref{F7b}. We see that the solution of the IMEX scheme is close to the reference ERKC solutions and no obvious oscillations are observed. In our test, the maximum error of discrete divergence of the magnetic field is $2.87 \times 10^{-11}$ for the IMEX scheme and $2.77 \times 10^{-11}$ for the ERKC scheme, both are close to machine precision.
\begin{figure}[hbtp]
\begin{center}
	\mbox{
		{\includegraphics[width=7cm]{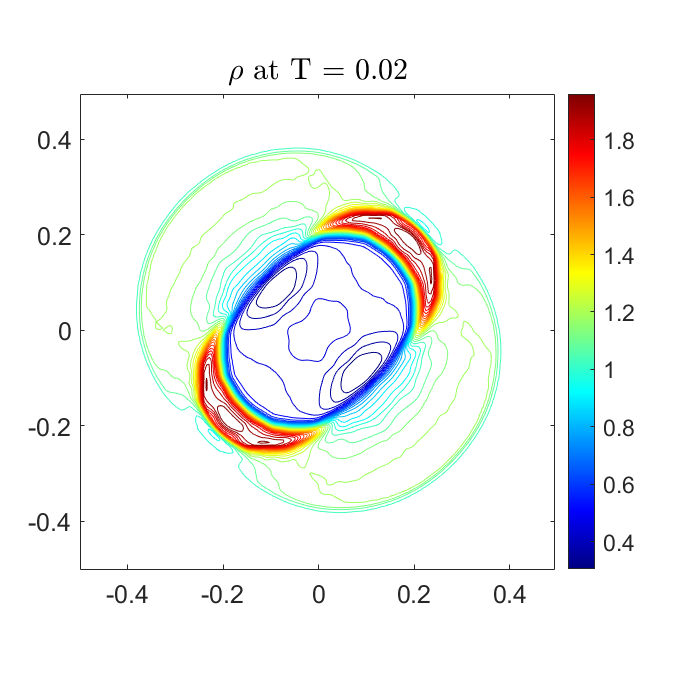}}
		{\includegraphics[width=7cm]{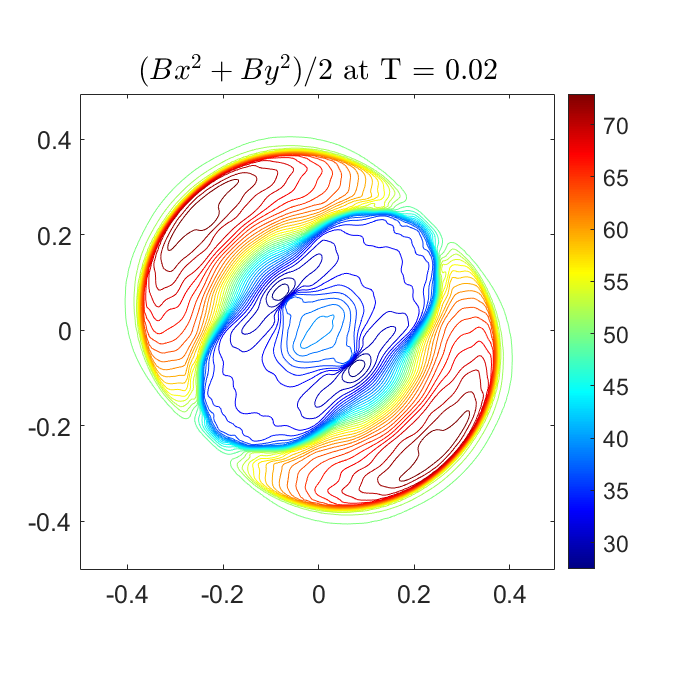}}
        }\\ \mbox{
		{\includegraphics[width=7cm]{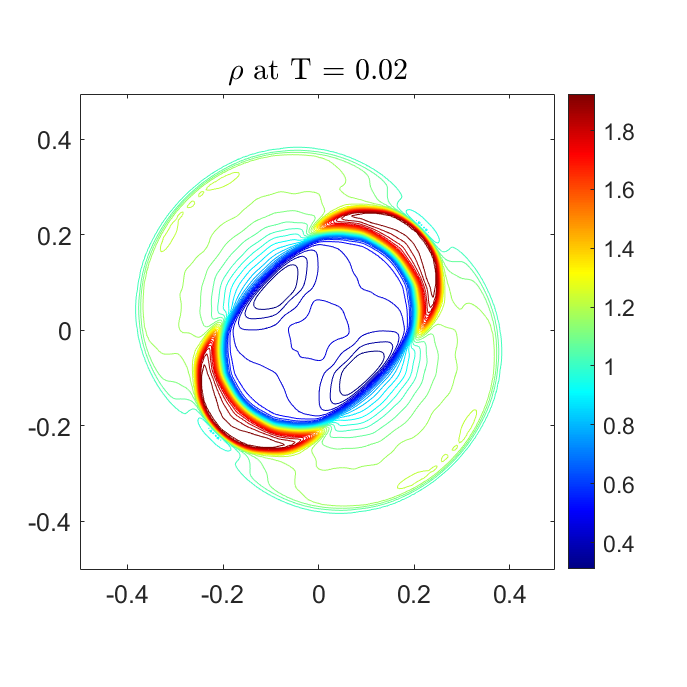}}
		{\includegraphics[width=7cm]{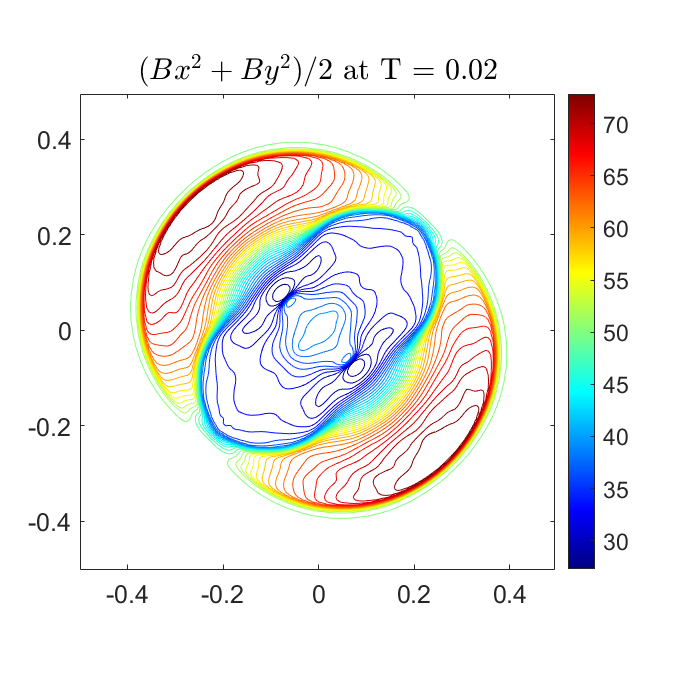}}
        }
\caption{ Example \ref{exam7}. MHD blast wave problem on the mesh grid $200^2$. Top: IMEX; Bottom: ERKC. 35 equally spaced contours are used.}
\label{F7a}
\end{center}
\end{figure}
\begin{figure}[hbtp]
\begin{center}
	\mbox{
		{\includegraphics[width=7.3cm]{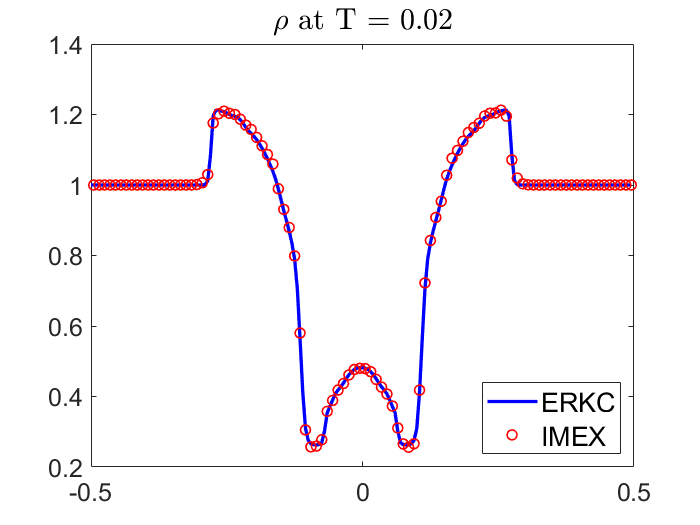}}
		{\includegraphics[width=7.3cm]{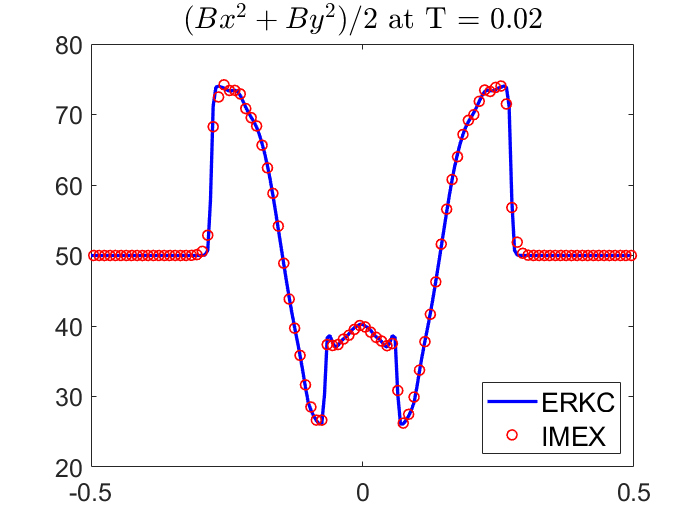}}
        }
\caption{ Example \ref{exam7} MHD blast wave problem. 1D cutting plots along $y=-x + 0.5$.  Every two points are plotted for the IMEX scheme for the cutting plots.}
\label{F7b}
\end{center}
\end{figure}
}
\end{exa}

\begin{exa} {\em
\label{exam8}
({\bf{Field loop advection}}) Now, we show the 2D field loop advection test \cite{mamashita2021slau2,minoshima2019high} to assess the capability of capturing a tangential discontinuity for a multidimensional flow. The computational domain is $\Omega = [-1,1] \times [-0.5,0.5]$ with periodic boundary conditions. The domain is divided with $256 \times 128$ uniform cells. The initial conditions are set as
\begin{equation*}
\begin{aligned}
        &(\rho,u,v,w,B_x,B_y,B_z,p)(r,0) \\
    &\quad = (1, v_0 \cos{(\theta_0)}, v_0 \sin{(\theta_0)}, 0, \partial A_z / \partial y, - \partial A_z / \partial x, 0, 1)
\end{aligned}
\end{equation*}
with $v_0 = 1, \cos{(\theta_0)} = -2 / \sqrt{5}, \sin{(\theta_0)} = 1 / \sqrt{5}, r = \sqrt{x^2+y^2}$, and
\begin{equation*}
\begin{aligned}       
    A_z(r) = \left\{
    \begin{aligned}
        &10^{-3}(0.3-r) \quad& r \le 0.3,&\\
        &0 \quad& otherwise.&
    \end{aligned}
    \right.
\end{aligned}
\end{equation*}
In this test, a loop of the magnetic field is advected. In Fig.~\ref{F8}, we show both the deviation of the density $\rho$ away from 1 and the magnetic pressure $\Vert \bB \Vert^2 / 2$ for the case of $\eps = 0.1$ at $T = \sqrt{5}$. For the magnetic pressure, both schemes preserve the initial symmetric loop structure very well, and the IMEX scheme has almost the same results as the ERKC scheme. The one-dimensional cutting plots for the magnetic pressure along $x = 0.2$ show that both schemes slightly distort the profile at $y = 0.15$. For the deviation of the density, we can find the IMEX scheme is better than the ERKC scheme for preserving a symmetric structure, as compared with the results of the ERKC scheme on a uniform mesh of $384 \times 192$ cells. For this problem in the intermediate regime with $\eps=0.1$, we also compare the CPU cost of two schemes in Table \ref{T8}. We can observe that the ERKC scheme costs less when $\eps$ is of $\mathcal{O}(1)$. However, the IMEX scheme is much more efficient in the intermediate and low sonic Mach regimes. The computational cost for the ERKC scheme increases very rapidly, as the sonic Mach number $\eps$ becomes smaller.
This example demonstrates the main advantage of an AP scheme 
for such multiscale problems.
\begin{figure}[hbtp]
\begin{center}
        \mbox{
		{\includegraphics[width=8cm]{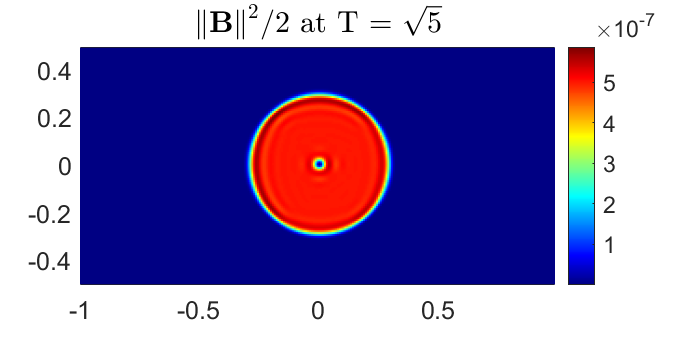}}
		{\includegraphics[width=8cm]{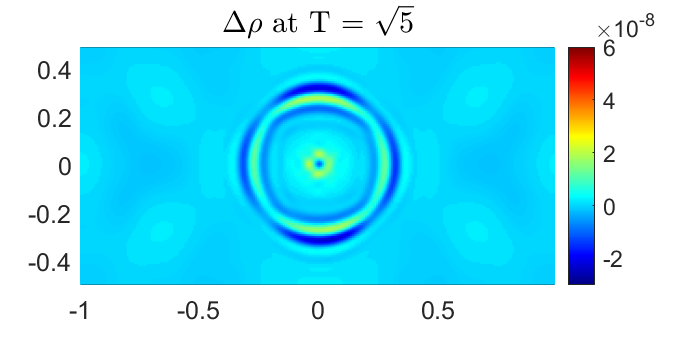}}
	    }\\ \mbox{
        {\includegraphics[width=8cm]{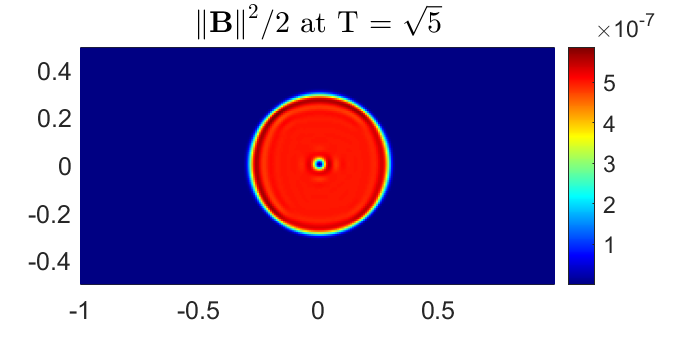}}
		{\includegraphics[width=8cm]{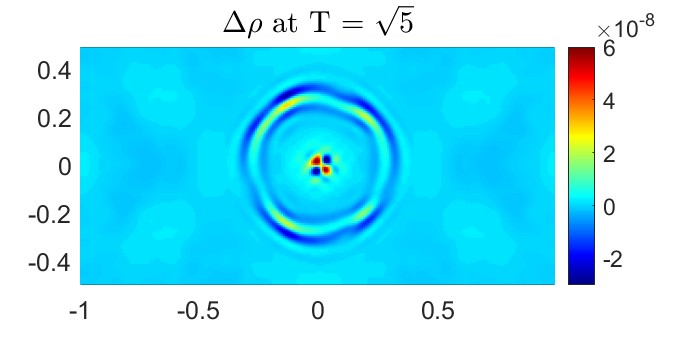}}
		}\\ \mbox{
        {\includegraphics[width=8cm]{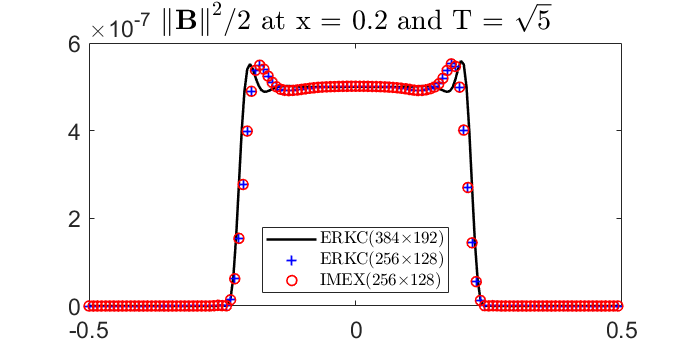}}
		{\includegraphics[width=8cm]{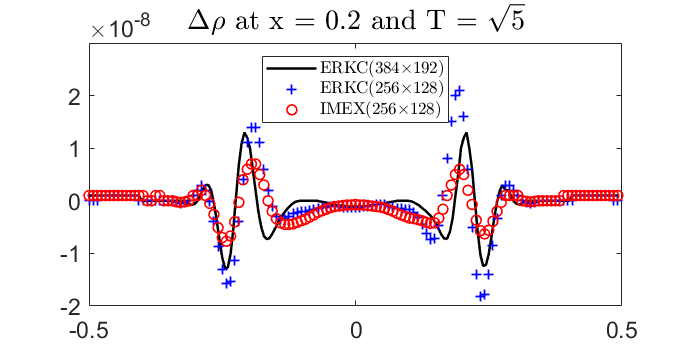}}
		}
\caption{ Example \ref{exam8}. Field loop advection on the mesh grid $256 \times 128$. Top left: magnetic pressure (IMEX); Middle left: magnetic pressure (ERKC); Bottom left: 1D cutting plot for the magnetic pressure; Top right: deviation of density (IMEX); Middle right: deviation of density (ERKC); Bottom right: 1D cutting plot for the deviation of density.}
\label{F8}
\end{center}
\end{figure}

\begin{table}[htbp]
  \caption{ Example \ref{exam8}. The CPU cost (seconds) for the IMEX and ERKC schemes on the mesh grid $256 \times 128$ at $T = \sqrt{5}$. }
  \begin{center}
  \begin{tabular}{|c|c|c|c|}\hline
   $\eps$ & 0.5 & 0.1 & 0.05   \\ \hline
   IMEX & 7634.880 & 5687.467 & 8015.959 \\ \hline
   ERKC & 4736.466 & 19719.964 & 38666.192  \\ \hline
  \end{tabular}
  \end{center}
  \label{T8}
\end{table}
}
\end{exa}

\begin{exa} {\em
\label{exam9}
({\bf{Magnetized Kelvin–Helmholtz instability}}) Finally, we consider a simulation of the magnetized Kelvin–Helmholtz instability problem\cite{boscarino2022high,leidi2022finite}. The initial conditions are
\begin{equation*}
    \begin{aligned}
        &(\rho,u,v,w,B_x,B_y,B_z,p)(x,y,0) \\
    &\quad = (\gamma, 1-2\eta(x), 0.1 \sin{(2 \pi x)}, 0, 0.1, 0, 0, 1),
\end{aligned}
\end{equation*}
with a magnetic potential $A_z(x,y,0) = 0.1 y$ on the domain $\Omega = [0,2] \times [-0.5,0.5]$, and
\begin{equation*}
\begin{aligned}       
    \eta(x) = \left\{
    \begin{aligned}
        &\frac{1}{2} \left(1 + \sin\left(16 \pi \left(y + \frac{1}{4}\right)\right)\right) \quad& y \in [-\frac{9}{32},-\frac{7}{32}],&\\
        &1 \quad& y \in [-\frac{7}{32},\enspace \; \frac{7}{32}],&\\
        &\frac{1}{2} \left(1 - \sin\left(16 \pi \left(y - \frac{1}{4}\right)\right)\right) \quad& y \in [\enspace \; \frac{7}{32},\enspace \; \frac{9}{32}],&\\
        &0 \quad& otherwise.&
    \end{aligned}
    \right.
\end{aligned}
\end{equation*}
Periodic boundary conditions are used for the conserved quantities, while a first-order extrapolation is utilized when the magnetic potential is outside of the computational domain. We run the solution up to $T = 0.8$ with $\gamma = 1.4$ and $\eps = 10^{-6}$ on a mesh grid $N_x \times N_y = 256 \times 128$. We define the local sonic Mach number as $M_{a} = \sqrt{u^2+v^2} / \sqrt{\gamma p / \rho}$, and the ratio of the local sonic Mach number $M_{ratio}$ to the maximum value of $M_a$ as $M_{ratio} = M_a/ \max(M_a)$, where the maximum is taken over all computational grid points. The vorticity $\omega = v_x - u_y$, where $v_x$ and $u_y$ are discretized by the 4th order central difference, and $M_{ratio}$ are shown in Fig.~\ref{F9}. We observe that they are comparable to the results in \cite{leidi2022finite,boscarino2022high}, namely, our IMEX scheme can capture the incompressible MHD system in the low sonic Mach number regime very well, which verifies the AP property of our scheme. 
\begin{figure}[hbtp]
\begin{center}
	\mbox{
		{\includegraphics[width=8cm]{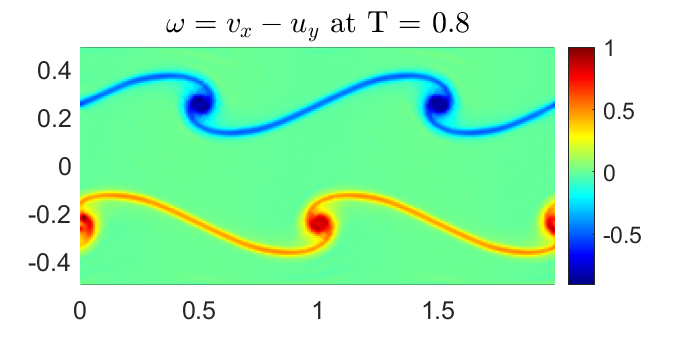}}
		{\includegraphics[width=8cm]{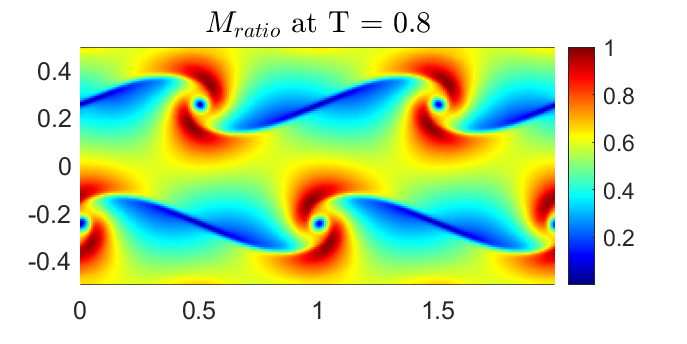}}
        }
\caption{ Magnetized Kelvin–Helmholtz instability problem for Example \ref{exam9} on the mesh grid $256 \times 128$. Left: Vorticity $\omega = v_x-u_y$; Right: The ratio of the
local sonic Mach number $M_{ratio}$.}
\label{F9}
\end{center}
\end{figure}
}
\end{exa}
%%%%%%%%%%%%%%%%%%%%%%%%%%%%%
%%%%%%%%%%%%%%%%%%%%%%%%%%%%%
%%%%%%%%%%%%%%%%%%%%%%%%%%%%%
%%%%%%%%%%%%%%%%%%%%%%%%%%%%%
%%%%%%%%%%%%%%%%%%%%%%%%%%%%%
%%%%%%%%%%%%%%%%%%%%%%%%%%%%%

\section{Conclusion}
\label{sec6}
\setcounter{equation}{0}
\setcounter{figure}{0}
\setcounter{table}{0}
In this paper, a high-order SI AP scheme with a discrete divergence-free property for the MHD equations has been developed. High-order accuracy in time has been obtained by SI temporal integrator based on an IMEX-RK framework. High-order accuracy in space has been achieved by finite difference WENO reconstructions. We have formally proved that the scheme is AP. We have also proved the AA property in the stiff limit as the sonic Mach number $\eps \rightarrow 0$, if the SI IMEX-RK scheme is SA. Numerical experiments in 1D and 2D have demonstrated the divergence-free for all ranges of sonic Mach numbers, AP and AA properties in the low sonic Mach limit. Compared to the
explicit schemes as reference solutions, such as the ERK \cite{jiang1999high} and ERKC \cite{christlieb2014finite} schemes in 1D and 2D respectively, our IMEX scheme has almost the same performances as the ERK or ERKC scheme, but is much more efficient in the intermediate and low sonic Mach regimes.

\bibliographystyle{abbrv}
\bibliography{refer}

%%%%%%%%%%%%%%%%%%%%%%%%%%%%%%%%%%
%
%%%%%%%%%%%%%%%%%%%%%%%%%%%%%%%%%%

\end{document}